\documentclass[onefignum,onetabnum]{siamart190516}

\usepackage{amsfonts}

\newsiamremark{remark}{Remark}

\headers{HJBI equations for fractional-order systems}{Mikhail I. Gomoyunov}

\title{On viscosity solutions of path-dependent Hamilton--Jacobi--Bellman--Isaacs equations for fractional-order systems\thanks{Submitted to the editors September 6, 2021.
\funding{This work was supported by the Russian Science Foundation, project~21-71-10070, \url{https://rscf.ru/en/project/21-71-10070/}}}}

\author{Mikhail I. Gomoyunov\thanks{N.N.~Krasovskii Institute of Mathematics and Mechanics of the Ural Branch of the Russian Academy of Sciences, Ekaterinburg, 620108, Russia, and Ural Federal University, Ekaterinburg, 620002, Russia (\email{m.i.gomoyunov@gmail.com})}}

\usepackage{amssymb}

\def\rd{\mathrm{d}}

\begin{document}

\maketitle

\begin{abstract}
    This paper deals with a two-person zero-sum differential game for a dynamical system described by a Caputo fractional differential equation of order $\alpha \in (0, 1)$ and a Bolza cost functional.
    The differential game is associated to the Cauchy problem for the path-dependent Hamilton--Jacobi--Bellman--Isaacs equation with so-called fractional coinvariant derivatives of the order $\alpha$ and the corresponding right-end boundary condition.
    A notion of a viscosity solution of the Cauchy problem is introduced, and the value functional of the differential game is characterized as a unique viscosity solution of this problem.
\end{abstract}



\section{Introduction}

    In this paper, we study a two-person zero-sum differential game (see, e.g.,~\cite{Isaacs_1965,Krasovskii_Subbotin_1988,Yong_2015,Cardaliaguet_Rainer_2018}) involving a dynamical system described by a Caputo fractional differential equation of order $\alpha \in (0, 1)$ (see, e.g.,~\cite{Samko_Kilbas_Marichev_1993,Kilbas_Srivastava_Trujillo_2006,Diethelm_2010}) and a Bolza cost functional, which the first player tries to minimize while the second player tries to maximize.
    In accordance with~\cite{Gomoyunov_2021_Mathematics,Gomoyunov_Lukoyanov_2021}, we associate the differential game to the Cauchy problem for the Hamilton--Jacobi--Bellman--Isaacs (HJBI) equation with so-called fractional coinvariant ($ci$-) derivatives of the order $\alpha$ (see, e.g.,~\cite{Kim_1999,Gomoyunov_2020_SIAM} and also~\cite{Gomoyunov_Lukoyanov_Plaksin_2021}) and the corresponding right-end boundary condition.
    It should be noted that the path-dependent nature of the Caputo fractional derivative makes it necessary to consider the value of the differential game as a non-anticipative functional on a certain space of paths, and, respectively, the HJBI equation can be classified as path-dependent (in this connection, see, e.g.,~\cite{Pham_Zhang_2014,Kaise_2015,Tang_Zhang_2015,Ekren_Touzi_Zhang_2016_1,Bayraktar_Keller_2018,Cosso_Russo_2019_Osaka,Saporito_2019,Gomoyunov_Lukoyanov_Plaksin_2021}).

    As stated in~\cite[Theorem~1]{Gomoyunov_2021_Mathematics}, if the value functional of the differential game under consideration is $ci$-smooth of the order $\alpha$ (which means that it is continuous and has continuous $ci$-derivatives of the order $\alpha$), then it is characterized as a unique solution of the associated Cauchy problem in a classical sense.
    However, the value functional usually does not have these smoothness properties, which leads to the need to introduce and investigate solutions of this Cauchy problem in a generalized sense.

    Following~\cite{Subbotin_1995} and, in a path-dependent framework,~\cite{Lukoyanov_2003_1} (see also~\cite{Gomoyunov_Lukoyanov_Plaksin_2021} and the references therein), previous works in this direction~\cite{Gomoyunov_2020_IMM_Eng,Gomoyunov_2021_IMM_Eng,Gomoyunov_2021_ESAIM} focused on the development of the minimax approach to a notion of a generalized solution.
    In particular, the questions of well-posedness of minimax solutions, consistency of minimax solutions with classical solutions, as well as and non-local and infinitesimal criteria of minimax solutions were addressed.
    Moreover, relying on these results, it was proved (see~\cite[Theorem~2]{Gomoyunov_2021_Mathematics} and also~\cite[Theorem~1]{Gomoyunov_Lukoyanov_2021}) that the value functional of the differential game under consideration coincides with a unique minimax solution of the associated Cauchy problem.

    In this paper, we apply another technique and, following~\cite{Crandall_Lions_1983} and, in a path-dependent framework,~\cite{Soner_1988,Lukoyanov_2007_IMM_Eng,Plaksin_2021}, develop the viscosity approach.
    Our main objectives are to propose an appropriate notion of a viscosity solution of the investigated Cauchy problem and to characterize the value functional of the considered differential game as a unique viscosity solution of this problem.

    More precisely, we give a definition of a viscosity solution of the Cauchy problem through $ci$-smooth of the order $\alpha$ test functionals and a sequence of compact subsets of the path space, each of which is strongly invariant with respect to the considered fractional-order dynamical system and the union of which covers the path space.
    Furthermore, we require that a viscosity solution satisfies a certain local Lipschitz continuity condition with respect to the path variable.

    Then, we show that the value functional of the differential game is a viscosity solution of the Cauchy problem.
    To this end, we use the facts that the value functional is a minimax solution of this problem by~\cite[Theorem~2]{Gomoyunov_2021_Mathematics} and that the value functional meets the additional Lipschitz continuity requirement due to~\cite[Lemma~1]{Gomoyunov_Lukoyanov_2021} and apply the arguments from~\cite[Theorem~1]{Lukoyanov_2007_IMM_Eng}.

    After that, we establish uniqueness of a viscosity solution of the Cauchy problem.
    In general, the proof of this result is carried out by the scheme from~\cite[Theorem~2]{Lukoyanov_2007_IMM_Eng}.
    However, the key part of the proof, which is the construction of appropriate test functionals, is different.
    This construction takes into account the peculiarities of dealing with the HJBI equation involving the fractional $ci$-derivatives and can be considered as the major contribution of the paper.
    It should be noted also that the Lipschitz continuity requirement is crucial for the given proof, since, similarly to, e.g.,~\cite[Lemma~7.6]{Plaksin_2021}, it allows us to derive certain boundedness properties of $ci$-gradients of the order $\alpha$ of the obtained test functionals.

    The paper is organized as follows.
    After some preliminaries in~\cref{section_Preliminaries}, we describe the differential game, formulate the associated Cauchy problem, and present the necessary properties of the value functional in~\cref{section_Differential_Game_and_HJBI_Equation}.
    We give the definition of a viscosity solution of the Cauchy problem and show that the value functional is such a viscosity solution in~\cref{section_Viscosity_Solutions}.
    The uniqueness theorem for viscosity solutions is proved in~\cref{section_Uniqueness_Theorem}.
    The proofs of auxiliary statements are relegated to~\cref{Appendix}.

\section{Preliminaries}
\label{section_Preliminaries}

    Let $n \in \mathbb{N}$, $T > 0$, and $\alpha \in (0, 1)$ be fixed.
    By $\|\cdot\|$ and $\langle \cdot, \cdot \rangle$, the Euclidean norm and inner product in $\mathbb{R}^n$ are denoted.
    Given $R \geq 0$, let $B_R$ stand for the closed ball in $\mathbb{R}^n$ centered at the origin of radius $R$.

    Following, e.g.,~\cite[Definition~2.3]{Samko_Kilbas_Marichev_1993}, we introduce the linear space $AC^\alpha$ of all functions $x \colon [0, T] \to \mathbb{R}^n$ each of which can be represented in the form
    \begin{equation} \label{x_f}
        x(\tau)
        = x(0) + \frac{1}{\Gamma(\alpha)} \int_{0}^{\tau} \frac{f(\xi)}{(\tau - \xi)^{1 - \alpha}} \, \rd \xi
        \quad \forall \tau \in [0, T]
    \end{equation}
    for some (Lebesgue) measurable and essentially bounded function $f \colon [0, T] \to \mathbb{R}^n$.
    In the right-hand side of equality~\cref{x_f}, the second term is the Riemann--Liouville fractional integral of the order $\alpha$ of the function $f(\cdot)$ (see, e.g.,~\cite[Definition~2.1]{Samko_Kilbas_Marichev_1993}) and $\Gamma$ is the gamma-function.
    Noting that every function $x(\cdot) \in AC^\alpha$ is continuous (see, e.g.,~\cite[Remark~3.3]{Samko_Kilbas_Marichev_1993}), we consider $AC^\alpha$ as a subspace of the Banach space of all continuous functions $x \colon [0, T] \to \mathbb{R}^n$ endowed with the uniform norm
    \begin{displaymath}
        \|x(\cdot)\|_\infty
        \triangleq \max_{\tau \in [0, T]} \|x(\tau)\|.
    \end{displaymath}

    According to, e.g.,~\cite[Theorem~2.4]{Samko_Kilbas_Marichev_1993}, every function $x(\cdot) \in AC^\alpha$ has at almost every (a.e.) $\tau \in [0, T]$ a Caputo fractional derivative of the order $\alpha$, which is defined by (see, e.g.,~\cite[Section~2.4]{Kilbas_Srivastava_Trujillo_2006} and~\cite[Chapter~3]{Diethelm_2010})
    \begin{equation} \label{Caputo}
        (^C D^\alpha x)(\tau)
        = \frac{1}{\Gamma(1 - \alpha)} \frac{\rd}{\rd \tau} \int_{0}^{\tau} \frac{x(\xi) - x(0)}{(\tau - \xi)^\alpha} \, \rd \xi.
    \end{equation}
    Moreover, if representation~\cref{x_f} is valid for some measurable and essentially bounded function $f(\cdot)$, then $(^C D^\alpha x)(\tau) = f(\tau)$ for a.e. $\tau \in [0, T]$.
    In particular, we have
    \begin{equation} \label{I^alpha_D^alpha}
        x(\tau)
        = x(0) + \frac{1}{\Gamma(\alpha)} \int_{0}^{\tau} \frac{(^C D^\alpha x)(\xi)}{(\tau - \xi)^{1 - \alpha}} \, \rd \xi
        \quad \forall \tau \in [0, T].
    \end{equation}

    Further, for every $(t, x(\cdot)) \in [0, T] \times AC^\alpha$, put
    \begin{equation} \label{Y}
        Y(t, x(\cdot))
        \triangleq \bigl\{ y(\cdot) \in AC^\alpha \colon
        \, y(\tau) = x(\tau)
        \quad \forall \tau \in [0, t] \bigr\}.
    \end{equation}

    A functional $\varphi \colon [0, T] \times AC^\alpha \to \mathbb{R}$ is called non-anticipative if the equality $\varphi(t, x(\cdot)) = \varphi(t, y(\cdot))$ holds for any $(t, x(\cdot)) \in [0, T) \times AC^\alpha$ and any $y(\cdot) \in Y(t, x(\cdot))$.

    A functional $\varphi \colon [0, T] \times AC^\alpha \to \mathbb{R}$ is said to be $ci$-differentiable of the order $\alpha$ at a point $(t, x(\cdot)) \in [0, T) \times AC^\alpha$ (see, e.g.,~\cite{Kim_1999,Gomoyunov_2020_SIAM} and also~\cite{Gomoyunov_Lukoyanov_Plaksin_2021}) if there exist $\partial_t^\alpha \varphi (t, x(\cdot)) \in \mathbb{R}$ and $\nabla^\alpha \varphi (t, x(\cdot)) \in \mathbb{R}^n$ such that, for any $y(\cdot) \in Y(t, x(\cdot))$ and any $\tau \in (t, T)$, the relation below takes place:
    \begin{align*}
        & \varphi(\tau, y(\cdot)) - \varphi(t, x(\cdot)) \\
        & \quad = \partial_t^\alpha \varphi(t, x(\cdot)) (\tau - t)
        + \int_{t}^{\tau} \langle \nabla^\alpha \varphi(t, x(\cdot)), (^C D^\alpha y)(\xi) \rangle \, \rd \xi
        + o(\tau - t),
    \end{align*}
    where the function $o \colon (0, \infty) \to \mathbb{R}$, which may depend on $t$ and $y(\cdot)$, satisfies the condition $o(\delta) / \delta \to 0$ as $\delta \to 0^+$.
    In this case, the values $\partial_t^\alpha \varphi (t, x(\cdot))$ and $\nabla^\alpha \varphi(t, x(\cdot))$ are called $ci$-derivatives of the order $\alpha$ of the functional $\varphi$ at the point $(t, x(\cdot))$.

    Observe that, if a functional $\varphi \colon [0, T] \times AC^\alpha \to \mathbb{R}$ is $ci$-differentiable of the order $\alpha$ at every point $(t, x(\cdot)) \in [0, T) \times AC^\alpha$, then the functional $\varphi$ itself and the mappings
    \begin{equation} \label{ci_derivatives}
        \partial_t^\alpha \varphi \colon [0, T) \times AC^\alpha \to \mathbb{R},
        \quad \nabla^\alpha \varphi \colon [0, T) \times AC^\alpha \to \mathbb{R}^n
    \end{equation}
    are automatically non-anticipative.

    Finally, a functional $\varphi \colon [0, T] \times AC^\alpha \to \mathbb{R}$ is said to be $ci$-smooth of the order $\alpha$ if it is continuous, $ci$-differentiable of the order $\alpha$ at every point $(t, x(\cdot)) \in [0, T) \times AC^\alpha$, and mappings~\cref{ci_derivatives} are continuous.

    \begin{remark}
        In this paper, we consider the space $[0, T] \times AC^\alpha$, endowed with the standard product metric, and require that all mappings defined on this space are non-anticipative.
        Note that this framework differs from previous studies~\cite{Gomoyunov_2020_SIAM,Gomoyunov_2020_IMM_Eng,Gomoyunov_2020_DE,Gomoyunov_2021_IMM_Eng,Gomoyunov_2021_Mathematics,Gomoyunov_2021_ESAIM,Gomoyunov_Lukoyanov_2021}, which deal with a certain metric space consisting of all pairs $(t, w(\cdot))$ where $t \in [0, T]$ and $w(\cdot)$ is a restriction of a function $x(\cdot) \in AC^\alpha$ to $[0, t]$.
        Nevertheless, as shown in~\cite[Subsection~5.1]{Gomoyunov_Lukoyanov_Plaksin_2021}, these two approaches are closely related, which allows us to use the previously obtained results in this paper, after their corresponding reformulation.
    \end{remark}

\section{Differential game and HJBI equation}
\label{section_Differential_Game_and_HJBI_Equation}

    In this section, we first describe shortly a differential game under consideration and formulate the basic assumptions in~\cref{subsection_Differential_Game}.
    Then, we associate the differential game to the Cauchy problem for the path-dependent HJBI equation with $ci$-derivatives of the order $\alpha$ and the corresponding right-end boundary condition in~\cref{subsection_HJBI_Equation}.
    After that, in~\cref{subsection_Properties_of_rho}, we recall a notion of a minimax solution of the Cauchy problem and present the properties of the value functional of the differential game that are needed for the main results of the paper.

\subsection{Differential game}
\label{subsection_Differential_Game}

    This paper deals with the following differential game.
    Given an initial data $(t, x(\cdot)) \in [0, T] \times AC^\alpha$, a path $y(\cdot) \in Y(t, x(\cdot))$ (see~\cref{Y}) of the dynamical system is described by the fractional differential equation
    \begin{subequations} \label{differential_game}
        \begin{equation} \label{system}
        (^C D^\alpha y)(\tau)
        = f(\tau, y(\tau), u(\tau), v(\tau))
        \text{ for a.e. } \tau \in [t, T].
        \end{equation}
        Here, $\tau$ is time, $y(\tau) \in \mathbb{R}^n$ is a current state, $(^C D^\alpha y)(\tau)$ is the Caputo fractional derivative of the order $\alpha$ (see~\cref{Caputo}), $u(\tau) \in P$ and $v(\tau) \in Q$ are current controls of the first and second players, respectively, where $P \subset \mathbb{R}^{n_P}$ and $Q \subset \mathbb{R}^{n_Q}$ are compact sets and $n_P$, $n_Q \in \mathbb{N}$.
        The first player tries to minimize while the second player tries to maximize the Bolza cost functional
        \begin{equation} \label{cost_functional}
            J
            \triangleq \sigma(y(\cdot)) + \int_{t}^{T} \chi(\tau, y(\tau), u(\tau), v(\tau)) \, \rd \tau.
        \end{equation}
    \end{subequations}

    The mappings $f \colon [0, T] \times \mathbb{R}^n \times P \times Q \to \mathbb{R}^n$, $\chi \colon [0, T] \times \mathbb{R}^n \times P \times Q \to \mathbb{R}$, and $\sigma \colon AC^\alpha \to \mathbb{R}$ are assumed to satisfy the conditions below:
    \begin{itemize}
        \item[(i)]
            The functions $f$ and $\chi$ are continuous.

        \item[(ii)]
            For any $R > 0$, there exists $\lambda_\ast > 0$ such that
            \begin{displaymath}
                \|f(t, x, u, v) - f(t, y, u, v)\|
                + |\chi(t, x, u, v) - \chi(t, y, u, v)|
                \leq \lambda_\ast \|x - y\|
            \end{displaymath}
            for any $t \in [0, T]$, any $x$, $y \in B_R$, any $u \in P$, and any $v \in Q$.

        \item[(iii)]
            There is a number $c_\ast > 0$ such that
            \begin{displaymath}
                \|f(t, x, u, v)\|
                \leq c_\ast (1 + \|x\|)
            \end{displaymath}
            for any $t \in [0, T]$, any $x \in \mathbb{R}^n$, any $u \in P$, and any $v \in Q$.

        \item[(iv)]
            For every $t \in [0, T]$ and every $x$, $s \in \mathbb{R}^n$, the following equality is valid:
            \begin{align*}
                & \min_{u \in P} \max_{v \in Q} \bigl( \langle s, f(t, x, u, v) \rangle + \chi(t, x, u, v) \bigr) \\
                & \quad = \max_{v \in Q} \min_{u \in P} \bigl( \langle s, f(t, x, u, v) \rangle + \chi(t, x, u, v) \bigr).
            \end{align*}

        \item[(v)]
            For any compact set $X \subset AC^\alpha$, there exists $\lambda^\ast > 0$ such that
            \begin{displaymath}
                |\sigma(x(\cdot)) - \sigma(y(\cdot))|
                \leq \lambda^\ast \Biggl( \|x(T) - y(T)\| + \int_{0}^{T} \|x(\tau) - y(\tau)\| \, \rd \tau \Biggr)
            \end{displaymath}
            for any $x(\cdot)$, $y(\cdot) \in X$.
    \end{itemize}
    Note that condition (v) implies continuity of the functional $\sigma$.

    For a complete statement of differential game~\cref{differential_game}, the reader is referred to~\cite[Sections~3 and~4]{Gomoyunov_2021_Mathematics}.
    We omit the details, since, for the purposes of this paper, we need only the properties of the value functional $[0, T] \times AC^\alpha \ni (t, x(\cdot)) \mapsto \rho(t, x(\cdot)) \in \mathbb{R}$ of differential game~\cref{differential_game} that have already been established in~\cite{Gomoyunov_2021_Mathematics,Gomoyunov_Lukoyanov_2021} and are presented below in~\cref{subsection_Properties_of_rho}.

\subsection{Path-dependent HJBI equation}
\label{subsection_HJBI_Equation}

    In accordance with~\cite{Gomoyunov_2021_Mathematics,Gomoyunov_Lukoyanov_2021}, we associate differential game~\cref{differential_game} to the Hamiltonian
    \begin{equation} \label{Hamiltonian}
        H(t, x, s)
        \triangleq \min_{u \in P} \max_{v \in Q} \bigl( \langle s, f(t, x, u, v) \rangle + \chi(t, x, u, v) \bigr),
        \quad t \in [0, T], \quad x, s \in \mathbb{R}^n,
    \end{equation}
    and the Cauchy problem for the path-dependent HJBI equation with fractional $ci$-derivatives of the order $\alpha$
    \begin{subequations} \label{Cauchy_problem}
        \begin{equation} \label{HJBI}
            \partial_t^\alpha \varphi(t, x(\cdot)) + H \bigl( t, x(t), \nabla^\alpha \varphi(t, x(\cdot)) \bigr)
            = 0
            \quad \forall (t, x(\cdot)) \in [0, T) \times AC^\alpha
        \end{equation}
        and the right-end boundary condition
        \begin{equation} \label{boundary_condition}
            \varphi(T, x(\cdot))
            = \sigma(x(\cdot))
            \quad \forall x(\cdot) \in AC^\alpha.
        \end{equation}
    \end{subequations}
    The unknown in problem~\cref{Cauchy_problem} is a non-anticipative functional $\varphi \colon [0, T] \times AC^\alpha \to \mathbb{R}$.

    Let us observe that assumptions (i)--(iii) imply the following properties of the Hamiltonian $H \colon [0, T] \times \mathbb{R}^n \times \mathbb{R}^n \to \mathbb{R}$ given by~\cref{Hamiltonian}:
    \begin{itemize}
        \item[(j)]
             The function $H$ is continuous.

        \item[(jj)]
            For every $t \in [0, T]$ and every $x$, $s$, $r \in \mathbb{R}^n$, the inequality below holds:
            \begin{displaymath}
                |H(t, x, s) - H(t, x, r)|
                \leq c_\ast (1 + \|x\|) \|s - r\|,
            \end{displaymath}
            where the number $c_\ast$ is taken from condition (iii).

        \item[(jjj)]
            For any $R > 0$, there exists $\lambda_\ast > 0$ such that
            \begin{displaymath}
                |H(t, x, s) - H(t, y, s)|
                \leq \lambda_\ast (1 + \|s\|) \|x - y\|
            \end{displaymath}
            for any $t \in [0, T]$, any $x$, $y \in B_R$, and any $s \in \mathbb{R}^n$.
        \end{itemize}

\subsection{Properties of value functional}
\label{subsection_Properties_of_rho}

    Given $(t, x(\cdot)) \in [0, T] \times AC^\alpha$, denote
    \begin{align}
        & Y_\ast(t, x(\cdot))
        \triangleq \Big\{ y(\cdot) \in Y(t, x(\cdot)) \colon \label{Y_ast} \\
        & \hphantom{Y_\ast(t, x(\cdot)) \triangleq \Big\{} \quad
        \|(^C D^\alpha y)(\tau)\| \leq c_\ast (1 + \|y(\tau)\|)
        \text{ for a.e. } \tau \in [t, T] \Bigr\}, \nonumber
    \end{align}
    where the set $Y(t, x(\cdot))$ is defined by~\cref{Y} and $c_\ast$ is the number from property (jj).

    According to~\cite[Subsection~3.3 and Proposition~4.3]{Gomoyunov_2021_ESAIM}, we have
    \begin{definition} \label{definition_minimax}
        A functional $\varphi \colon [0, T] \times AC^\alpha \to \mathbb{R}$ is called a minimax solution of Cauchy problem~\cref{Cauchy_problem} if it is non-anticipative and continuous, satisfies boundary condition~\cref{boundary_condition}, and possesses the following two properties:
        \begin{itemize}
            \item[$(\rm M_+)$]
                For any $(t, x(\cdot)) \in [0, T) \times AC^\alpha$ and any $s \in \mathbb{R}^n$, there exists $y(\cdot) \in Y_\ast(t, x(\cdot))$ such that
                \begin{displaymath}
                    \varphi(\tau, y(\cdot))
                    - \int_{t}^{\tau} \bigl( \langle s, (^C D^\alpha y)(\xi) \rangle - H(\xi, y(\xi), s) \bigr) \, \rd \xi
                    \leq \varphi(t, x(\cdot))
                    \quad \forall \tau \in [t, T].
                \end{displaymath}

            \item[$(\rm M_-)$]
                For any $(t, x(\cdot)) \in [0, T) \times AC^\alpha$ and any $s \in \mathbb{R}^n$, there exists $y(\cdot) \in Y_\ast(t, x(\cdot))$ such that
                \begin{displaymath}
                    \varphi(\tau, y(\cdot))
                    - \int_{t}^{\tau} \bigl( \langle s, (^C D^\alpha y)(\xi) \rangle - H(\xi, y(\xi), s) \bigr) \, \rd \xi
                    \geq \varphi(t, x(\cdot))
                    \quad \forall \tau \in [t, T].
                \end{displaymath}
        \end{itemize}
    \end{definition}

    By virtue of~\cite[Theorem~2 and Remark~1]{Gomoyunov_2021_Mathematics}, we obtain
    \begin{proposition} \label{proposition_the_value_is_the_minimax}
        Let assumptions {\rm (i)--(iv)} hold and let the functional $\sigma$ be continuous.
        Then, the value functional $\rho$ of differential game~\cref{differential_game} is a unique minimax solution of Cauchy problem~\cref{Cauchy_problem} where the Hamiltonian $H$ is given by~\cref{Hamiltonian}.
    \end{proposition}

    Further, for every $(t, x(\cdot)) \in [0, T] \times AC^\alpha$, we introduce the function
    \begin{equation} \label{a}
        a(\tau \mid t, x(\cdot))
        \triangleq
        \begin{cases}
            x(\tau),
            & \mbox{if } \tau \in [0, t], \\
            \displaystyle x(0) + \frac{1}{\Gamma(\alpha)} \int_{0}^{t} \frac{(^C D^\alpha x)(\xi)}{(\tau - \xi)^{1 - \alpha}} \, \rd \xi,
            & \mbox{if } \tau \in (t, T].
        \end{cases}
    \end{equation}
    Note that the function $a(\cdot \mid t, x(\cdot))$ can be defined as a unique function $a(\cdot) \in Y(t, x(\cdot))$ such that $(^C D^\alpha a)(\tau) = 0$ for a.e. $\tau \in [t, T]$.
    Hence, we derive
    \begin{equation} \label{semigroup_property}
        a(\cdot \mid t, x(\cdot))
        = a \bigl( \cdot \bigm| \tau, a(\cdot \mid t, x(\cdot)) \bigl)
        \quad \forall \tau \in [0, T].
    \end{equation}

    In addition, observe that, due to~\cite[Lemma~3]{Gomoyunov_2020_DE}, the mapping
    \begin{equation} \label{a_mapping}
        [0, T] \times AC^\alpha \ni (t, x(\cdot)) \mapsto a(\cdot \mid t, x(\cdot)) \in AC^\alpha
    \end{equation}
    is continuous, and, moreover, it holds that (see, e.g.,~\cite[the end of Section~6]{Gomoyunov_Lukoyanov_2021})
    \begin{equation} \label{a_Lipschitz}
        \|a(\cdot \mid t, x(\cdot))\|_\infty
        \leq \max_{\tau \in [0, t]} \|x(\tau)\|
        \quad \forall (t, x(\cdot)) \in [0, T] \times AC^\alpha.
    \end{equation}
    In particular, mapping~\cref{a_mapping} is non-anticipative, which means that
    \begin{equation} \label{a_non-anticipative}
        a(\cdot \mid t, y(\cdot))
        = a(\cdot \mid t, x(\cdot))
    \end{equation}
    for any $(t, x(\cdot)) \in [0, T) \times AC^\alpha$ and any $y(\cdot) \in Y(t, x(\cdot))$.

    Now, for a functional $\varphi \colon [0, T] \times AC^\alpha \to \mathbb{R}$, let us consider the following local Lipschitz continuity condition with respect to the second variable:
    \begin{itemize}
        \item[(L)]
            For any compact set $X \subset AC^\alpha$, there exists $\Lambda > 0$ such that
            \begin{align*}
                & |\varphi(t, x(\cdot)) - \varphi(t, y(\cdot))|
                \leq \Lambda \biggl( \|a(T \mid t, x(\cdot)) - a(T \mid t, y(\cdot))\| \\
                & \hphantom{|\varphi(t, x(\cdot)) - \varphi(t, y(\cdot))|}
                \quad + \int_{0}^{T} \frac{\|a(\tau \mid t, x(\cdot)) - a(\tau \mid t, y(\cdot))\|}{(T - \tau)^{1 - \alpha}} \, \rd \tau \biggr)
            \end{align*}
            for any $t \in [0, T]$ and any $x(\cdot)$, $y(\cdot) \in X$.
    \end{itemize}
    Owing to~\cref{a_non-anticipative}, condition (L) implies non-anticipativity of the functional $\varphi$.

    \begin{proposition} \label{proposition_the_value_is_Lipschitz}
        Under assumptions {\rm (i)--(v)}, the value functional $\rho$ of differential game~\cref{differential_game} satisfies local Lipschitz continuity condition $(\rm L)$.
    \end{proposition}

    This proposition can be proved by adapting the arguments from~\cite[Lemma~1]{Gomoyunov_Lukoyanov_2021} to the case of cost functional~\cref{cost_functional}, containing the additional integral term.

    \begin{remark} \label{remark_limiting_case_Lipschitz}
        In the limiting case $\alpha = 1$, the space $AC^1$ consists of all Lipschitz continuous functions $x \colon [0, T] \to \mathbb{R}^n$, and, for any $(t, x(\cdot)) \in [0, T] \times AC^1$, we have
        \begin{equation} \label{a_limiting_case}
            a(\tau \mid t, x(\cdot))
            =
            \begin{cases}
                x(\tau), & \mbox{if } \tau \in [0, t], \\[0.2em]
                x(t), & \mbox{if } \tau \in (t, T].
            \end{cases}
        \end{equation}
        Consequently, the inequality in condition (L) becomes as follows:
        \begin{displaymath}
            |\varphi(t, x(\cdot)) - \varphi(t, y(\cdot))|
            \leq \Lambda \biggl( (1 + T - t) \|x(t) - y(t)\| + \int_{0}^{t} \|x(\tau) - y(\tau)\| \, \rd \tau \biggr),
        \end{displaymath}
        which agrees with Lipschitz continuity conditions used in context of differential games for time-delay systems and the associated HJBI equations in, e.g.,~\cite{Lukoyanov_2010_IMM_Eng_2,Plaksin_2021}.
    \end{remark}

\section{Viscosity solutions}
\label{section_Viscosity_Solutions}

    For every $k \in \mathbb{N}$, put
    \begin{align*}
        & X_k
        \triangleq \Bigl\{ x(\cdot) \in AC^\alpha \colon
        \, \|x(0)\|
        \leq k, \\
        & \hphantom{X_k \triangleq \Bigl\{ x(\cdot) \in AC^\alpha \colon \,} \|(^C D^\alpha x)(\tau)\|
        \leq k c_\ast (1 + \|x(\tau)\|) \text{ for a.e. } \tau \in [0, T] \Bigr\},
    \end{align*}
    where the number $c_\ast$ is taken from property (jj) (respectively, from condition (iii)).
    According to, e.g.,~\cite[Theorem~2]{Gomoyunov_2020_DE}, the set $X_k$ is compact in $AC^\alpha$.
    Note that, for any $(t, x(\cdot)) \in [0, T] \times X_k$ and any $y(\cdot) \in Y_\ast(t, x(\cdot))$ (see~\cref{Y_ast}), the inclusion
    \begin{equation} \label{continuation_is_in_X_k}
        y(\cdot) \in X_k
    \end{equation}
    is valid, and, in addition, the equality
    \begin{equation} \label{X_k_union}
        AC^\alpha
        = \bigcup_{k \in \mathbb{N}} X_k
    \end{equation}
    takes place.

    \begin{definition} \label{definition_viscosity}
        A functional $\varphi \colon [0, T] \times AC^\alpha \to \mathbb{R}$ is called a viscosity solution of Cauchy problem~\cref{Cauchy_problem} if it is continuous, satisfies local Lipschitz continuity condition {\rm (L)}, meets boundary condition~\cref{boundary_condition}, and has the following two properties:
        \begin{itemize}
            \item[$(\rm V_+)$]
                For any $ci$-smooth of the order $\alpha$ test functional $\psi \colon [0, T] \times AC^\alpha \to \mathbb{R}$ and any $k \in \mathbb{N}$, if the difference $\varphi - \psi$ attains its minimum on the set $[0, T] \times X_k$ at some point $(t, x(\cdot)) \in [0, T) \times X_k$, then
                \begin{displaymath}
                    \partial_t^\alpha \psi(t, x(\cdot))
                    + H \bigl( t, x(t), \nabla^\alpha \psi(t, x(\cdot)) \bigr)
                    \leq 0.
                \end{displaymath}
            \item[$(\rm V_-)$]
                For any $ci$-smooth of the order $\alpha$ test functional $\psi \colon [0, T] \times AC^\alpha \to \mathbb{R}$ and any $k \in \mathbb{N}$, if the difference $\varphi - \psi$ attains its maximum on the set $[0, T] \times X_k$ at some point $(t, x(\cdot)) \in [0, T) \times X_k$, then
                 \begin{displaymath}
                    \partial_t^\alpha \psi(t, x(\cdot))
                    + H \bigl( t, x(t), \nabla^\alpha \psi(t, x(\cdot)) \bigr)
                    \geq 0.
                 \end{displaymath}
        \end{itemize}
    \end{definition}

    Our first result is
    \begin{theorem} \label{theorem_existence}
        Under assumptions $(\rm i)$--$(\rm v)$, the value functional $\rho$ of differential game~\cref{differential_game} is a viscosity solution of Cauchy problem~\cref{Cauchy_problem} where the Hamiltonian $H$ is given by~\cref{Hamiltonian}.
    \end{theorem}
    \begin{proof}
        By virtue of~\cref{definition_minimax,proposition_the_value_is_the_minimax}, the functional $\rho$ is non-anticipative and continuous, meets boundary condition~\cref{boundary_condition}, and possesses properties $(\rm M_+)$ and $(\rm M_-)$.
        Taking relation~\cref{continuation_is_in_X_k} into account and repeating the arguments from~\cite[Theorem~1]{Lukoyanov_2007_IMM_Eng}, we derive that properties $(\rm M_+)$ and $(\rm M_-)$ imply respectively properties $(\rm V_+)$ and $(\rm V_-)$.
        Thus, it remains to note that the functional $\rho$ satisfies condition $(\rm L)$ due to~\cref{proposition_the_value_is_Lipschitz}.
        The theorem is proved.
    \end{proof}

\section{Uniqueness of viscosity solutions}
\label{section_Uniqueness_Theorem}

    In this section, we prove
    \begin{theorem} \label{theorem_uniqueness}
        Let a Hamiltonian $H \colon [0, T] \times \mathbb{R}^n \times \mathbb{R}^n \to \mathbb{R}$ possess properties {\rm (j)} and {\rm (jj)}.
        Then, Cauchy problem~\cref{Cauchy_problem} admits at most one viscosity solution.
    \end{theorem}

    Before going into details, we note that~\cref{theorem_existence,theorem_uniqueness} yield the desired characterization of the value functional $\rho$ of differential game~\cref{differential_game}.
    Namely, we have
    \begin{corollary}
        Let assumptions {\rm (i)--(v)} hold.
        Then, the value functional $\rho$ of differential game~\cref{differential_game} is a unique viscosity solution of Cauchy problem~\cref{Cauchy_problem} with the Hamiltonian $H$ given by~\cref{Hamiltonian}.
    \end{corollary}

    The proof of~\cref{theorem_uniqueness} is carried out by the scheme from~\cite[Theorem~2]{Lukoyanov_2007_IMM_Eng} but with a different choice of an auxiliary functional needed for the construction of appropriate test functionals.
    Furthermore, the additional local Lipschitz continuity requirement in~\cref{definition_viscosity} is used similarly to, e.g.,~\cite[Lemma~7.6]{Plaksin_2021} in order to derive certain boundedness properties of $ci$-gradients of the order $\alpha$ of the obtained test functionals.
    Below, in~\cref{subsection_Lemmas}, we introduce the auxiliary functional and describe its properties.
    After that, in~\cref{subsection_Proof_Uniqueness}, we prove~\cref{theorem_uniqueness}.

    \subsection{Auxiliary functional}
    \label{subsection_Lemmas}

        Denote $q \triangleq 2 / (2 - \alpha) \in (1, 2)$ and take $\beta > 0$ such that $\beta < 1 - \alpha$ and $\beta < \alpha / 2$.
        For every number $\varepsilon > 0$, consider a functional $\nu_\varepsilon \colon [0, T] \times AC^\alpha \times [0, T] \times AC^\alpha \to \mathbb{R}$ given by
        \begin{align}
            & \nu_\varepsilon(t, x(\cdot), \tau, y(\cdot)) \label{nu} \\[0.3em]
            & \quad \triangleq \bigl( \varepsilon^{\frac{2}{q - 1}} + \|a(T \mid t, x(\cdot)) - a(T \mid \tau, y(\cdot))\|^2 \bigr)^{\frac{q}{2}} \nonumber \\[0.1em]
            & \qquad
            + \int_{0}^{T} \frac{\bigl( \varepsilon^{\frac{2}{q - 1}} + \|a(\xi \mid t, x(\cdot)) - a(\xi \mid \tau, y(\cdot))\|^2 \bigr)^{\frac{q}{2}}}
            {(T - \xi)^{(1 - \alpha - \beta) q}} \, \rd \xi
            - C_1 \varepsilon^{\frac{q}{q - 1}} \nonumber
        \end{align}
        for all $(t, x(\cdot))$, $(\tau, y(\cdot)) \in [0, T] \times AC^\alpha$, where the functions $a(\cdot \mid t, x(\cdot))$ and $a(\cdot \mid \tau, y(\cdot))$ are defined according to~\cref{a} and the number $C_1$ is as follows:
        \begin{displaymath}
            C_1
            \triangleq 1 + \frac{T^{1 - (1 - \alpha - \beta) q}}{1 - (1 - \alpha - \beta) q}.
        \end{displaymath}
        Note that
        \begin{displaymath}
            (1 - \alpha - \beta) q
            < (1 - \alpha) q
            = 1 - \frac{\alpha}{2 - \alpha}
            < 1.
        \end{displaymath}
        Hence, in particular, we have $C_1 > 0$ and, in view of continuity of the functions $a(\cdot \mid t, x(\cdot))$ and $a(\cdot \mid \tau, y(\cdot))$, the integral term in~\cref{nu} is well-defined.

        \begin{remark} \label{remark_limiting_case_functional}
            In the limiting case $\alpha = 1$, we obtain $q = 2$, $\beta = 0$, and
            \begin{align*}
                & \nu_\varepsilon(t, x(\cdot), \tau, y(\cdot))
                = \|a(T \mid t, x(\cdot)) - a(T \mid \tau, y(\cdot))\|^2 \\[0.1em]
                & \hphantom{\nu_\varepsilon(t, x(\cdot), \tau, y(\cdot))} \quad
                + \int_{0}^{T} \|a(\xi \mid t, x(\cdot)) - a(\xi \mid \tau, y(\cdot))\|^2 \, \rd \xi
            \end{align*}
            for all $\varepsilon > 0$ and all $(t, x(\cdot))$, $(\tau, y(\cdot)) \in [0, T] \times AC^1$.
            Hence, by virtue of~\cref{a_limiting_case}, the functional $\nu_\varepsilon$ turns into the functional proposed in the proof of~\cite[Theorem~2]{Lukoyanov_2007_IMM_Eng}.
        \end{remark}

        The five lemmas below establish the properties of the functional $\nu_\varepsilon$ that are used in the proof of~\cref{theorem_uniqueness}, sometimes without explicit reference.
        \begin{lemma} \label{lemma_nu_is_continuous}
            For any $\varepsilon > 0$, the functional $\nu_\varepsilon$ is non-negative and continuous.
            In addition, the equalities
            \begin{equation} \label{lemma_nu_properties_item_a}
                \nu_\varepsilon(t, x(\cdot), t, x(\cdot))
                = 0,
                \quad \nu_\varepsilon(t, x(\cdot), \tau, y(\cdot))
                = \nu_\varepsilon(\tau, y(\cdot), t, x(\cdot)),
            \end{equation}
            and
            \begin{equation} \label{lemma_nu_properties_item_b}
                \nu_\varepsilon \bigl( t^\prime, a(\cdot \mid t, x(\cdot)), \tau^\prime, a(\cdot \mid \tau, y(\cdot)) \bigr)
                = \nu_\varepsilon(t, x(\cdot), \tau, y(\cdot))
            \end{equation}
            are valid for any $(t, x(\cdot))$, $(\tau, y(\cdot)) \in [0, T] \times AC^\alpha$, any $t^\prime \in [t, T]$, and any $\tau^\prime \in [\tau, T]$.
        \end{lemma}
        \begin{proof}
            See~\cref{Appendix}.
        \end{proof}

        The next result provides a connection between the functional $\nu_\varepsilon$ and local Lipschitz continuity condition $(\rm L)$.
        \begin{lemma} \label{lemma_Lipschitz_via_nu}
            There exists a number $C_2 > 0$ such that
            \begin{align}
                & \|a(T \mid t, x(\cdot)) - a(T \mid \tau, y(\cdot))\|
                + \int_{0}^{T} \frac{\|a(\xi \mid t, x(\cdot)) - a(\xi \mid \tau, y(\cdot))\|}{(T - \xi)^{1 - \alpha}} \, \rd \xi \label{Lipschitz_via_nu_main} \\[0.2em]
                & \quad \leq C_2 \bigl( \nu_\varepsilon(t, x(\cdot), \tau, y(\cdot)) + C_1 \varepsilon^{\frac{q}{q - 1}} \bigr)^{\frac{1}{q}} \nonumber
            \end{align}
            for any $\varepsilon > 0$ and any $(t, x(\cdot))$, $(\tau, y(\cdot)) \in [0, T] \times AC^\alpha$.
        \end{lemma}
        \begin{proof}
            See~\cref{Appendix}.
        \end{proof}

        In addition, the following convergence property takes place.
        \begin{lemma} \label{lemma_convergene}
            Let a compact set $X \subset AC^\alpha$ be fixed and let $(t_\varepsilon, x_\varepsilon(\cdot))$, $(\tau_\varepsilon, y_\varepsilon(\cdot)) \in [0, T] \times X$ be given for every $\varepsilon > 0$.
            Suppose that $\nu_\varepsilon(t_\varepsilon, x_\varepsilon(\cdot), \tau_\varepsilon, y_\varepsilon(\cdot)) \to 0$ as $\varepsilon \to 0^+$.
            Then, it holds that
            \begin{equation} \label{convergence}
                \|a(\cdot \mid t_\varepsilon, x_\varepsilon(\cdot)) - a(\cdot \mid \tau_\varepsilon, y_\varepsilon(\cdot))\|_\infty
                \to 0
                \text{ as } \varepsilon \to 0^+.
            \end{equation}
        \end{lemma}
        \begin{proof}
            See~\cref{Appendix}.
        \end{proof}

        Further, for every $\varepsilon > 0$ and every $(\tau_\ast, y_\ast(\cdot)) \in [0, T] \times AC^\alpha$, consider a functional $\mu_\varepsilon^{(\tau_\ast, y_\ast(\cdot))} \colon [0, T] \times AC^\alpha \to \mathbb{R}$ defined by
        \begin{equation} \label{psi_varepsilon}
            \mu_\varepsilon^{(\tau_\ast, y_\ast(\cdot))} (t, x(\cdot))
            \triangleq \nu_\varepsilon(t, x(\cdot), \tau_\ast, y_\ast(\cdot))
        \end{equation}
        for all $(t, x(\cdot)) \in [0, T] \times AC^\alpha$.

        \begin{lemma} \label{lemma_ci_smoothness}
            For any $\varepsilon > 0$ and any $(\tau_\ast, y_\ast(\cdot)) \in [0, T] \times AC^\alpha$, the functional $\mu_\varepsilon^{(\tau_\ast, y_\ast(\cdot))}$ is $ci$-smooth of the order $\alpha$ and its $ci$-derivatives of the order $\alpha$ are given by
            \begin{equation} \label{lemma_nu_properties_item_c_derivative_t}
                \partial_t^\alpha \mu_\varepsilon^{(\tau_\ast, y_\ast(\cdot))} (t, x(\cdot))
                = 0
            \end{equation}
            and
            \begin{align}
                & \nabla^\alpha \mu_\varepsilon^{(\tau_\ast, y_\ast(\cdot))} (t, x(\cdot)) \label{lemma_nu_properties_item_c_derivative_x} \\[0.3em]
                & \ \ = \frac{q}{\Gamma(\alpha)} \biggl( \frac{a(T \mid t, x(\cdot)) - a_\ast(T)}
                {\bigl( \varepsilon^{\frac{2}{q - 1}} + \|a(T \mid t, x(\cdot)) - a_\ast(T)\|^2 \bigr)^{1 - \frac{q}{2}}
                (T - t)^{1 - \alpha}} \nonumber \\[0.1em]
                & \ \ \ \ + \int_{t}^{T} \frac{a(\xi \mid t, x(\cdot)) - a_\ast(\xi)}
                {\bigl( \varepsilon^{\frac{2}{q - 1}} + \|a(\xi \mid t, x(\cdot)) - a_\ast(\xi)\|^2 \bigr)^{1 - \frac{q}{2}}
                (\xi - t)^{1 - \alpha} (T - \xi)^{(1 - \alpha - \beta) q}} \, \rd \xi \biggr) \nonumber
            \end{align}
            for all $(t, x(\cdot)) \in [0, T) \times AC^\alpha$, where we denote $a_\ast(\cdot) \triangleq a(\cdot \mid \tau_\ast, y_\ast(\cdot))$.
        \end{lemma}
        \begin{proof}
            See~\cref{Appendix}.
        \end{proof}

        Let us observe that formula~\cref{lemma_nu_properties_item_c_derivative_x} itself is not needed for the proof of~\cref{theorem_uniqueness}, but it allows us to obtain
        \begin{lemma} \label{lemma_gradient_properties}
            For any $\theta \in (0, T)$, there are $C_3 > 0$ and $C_4 > 0$ such that
            \begin{equation} \label{gradient_bound}
                \| \nabla^\alpha \mu_\varepsilon^{(\tau, y(\cdot))} (t, x(\cdot)) \|
                \leq C_3 \bigl( \nu_\varepsilon(t, x(\cdot), \tau, y(\cdot)) + C_1 \varepsilon^{\frac{q}{q - 1}} \bigr)^{\frac{q - 1}{q}}
            \end{equation}
            and
            \begin{align}
                & \| \nabla^\alpha \mu_\varepsilon^{(\tau, y(\cdot))} (t, x(\cdot))
                + \nabla^\alpha \mu_\varepsilon^{(t, x(\cdot))} (\tau, y(\cdot)) \| \label{difference_of_gradients} \\[0.1em]
                & \quad \leq C_4 \bigl( \varepsilon^{\frac{2}{q - 1}} + \|a(\cdot \mid t, x(\cdot)) - a(\cdot \mid \tau, y(\cdot))\|_\infty^2 \bigr)^{\frac{q - 1}{2}}
                |t - \tau|^\alpha \nonumber
            \end{align}
            for any $\varepsilon > 0$ and any $(t, x(\cdot))$, $(\tau, y(\cdot)) \in [0, T - \theta] \times AC^\alpha$.
        \end{lemma}
        \begin{proof}
            See~\cref{Appendix}.
        \end{proof}

    \subsection{Proof of~\cref{theorem_uniqueness}}
    \label{subsection_Proof_Uniqueness}

        Suppose that $\varphi_1$ and $\varphi_2$ are two viscosity solutions of Cauchy problem~\cref{Cauchy_problem}.
        In view of~\cref{X_k_union}, to prove the theorem, it suffices to show that, for every $k \in \mathbb{N}$,
        \begin{displaymath}
            \varphi_1 (t, x(\cdot))
            \leq \varphi_2 (t, x(\cdot))
            \quad \forall (t, x(\cdot)) \in [0, T] \times X_k.
        \end{displaymath}
        Arguing by contradiction, assume that there exists $k \in \mathbb{N}$ such that
        \begin{equation} \label{proof_choice_of_varkappa}
            \varkappa
            \triangleq \max_{(t, x(\cdot)) \in [0, T] \times X_k} \bigl( \varphi_1(t, x(\cdot)) - \varphi_2(t, x(\cdot)) \bigr)
            > 0.
        \end{equation}
        Note that the maximum is attained due to continuity of the functionals $\varphi_1$ and $\varphi_2$ and compactness of the set $[0, T] \times X_k$.

        For every $\varepsilon > 0$, define a functional $\Phi_\varepsilon \colon [0, T] \times AC^\alpha \times [0, T] \times AC^\alpha \to \mathbb{R}$ by
        \begin{align*}
            & \Phi_\varepsilon(t, x(\cdot), \tau, y(\cdot))
            = \varphi_1(t, x(\cdot)) - \varphi_2(\tau, y(\cdot)) \\[0.2em]
            & \hphantom{\Phi_\varepsilon(t, x(\cdot), \tau, y(\cdot))} \quad
            - (2 T - t - \tau) \zeta - \frac{(t - \tau)^2}{\varepsilon^{\frac{3}{\alpha}}}
            - \frac{\nu_\varepsilon(t, x(\cdot), \tau, y(\cdot))}{\varepsilon}
        \end{align*}
        for all $(t, x(\cdot))$, $(\tau, y(\cdot)) \in [0, T] \times AC^\alpha$, where
        \begin{equation} \label{proof_choice_of_zeta}
            \zeta
            \triangleq \frac{\varkappa}{4 T}
            > 0
        \end{equation}
        and the functional $\nu_\varepsilon$ is given by~\cref{nu}.
        Taking into account that the functional $\nu_\varepsilon$ is continuous (see~\cref{lemma_nu_is_continuous}), choose $(t_\varepsilon, x_\varepsilon(\cdot))$, $(\tau_\varepsilon, y_\varepsilon(\cdot)) \in [0, T] \times X_k$ such that
        \begin{equation} \label{proof_max}
            \Phi_\varepsilon(t_\varepsilon, x_\varepsilon(\cdot), \tau_\varepsilon, y_\varepsilon(\cdot))
            = \max_{(t, x(\cdot)), (\tau, y(\cdot)) \in [0, T] \times X_k} \Phi_\varepsilon(t, x(\cdot), \tau, y(\cdot)).
        \end{equation}

        For every $\varepsilon > 0$, since $\varphi_1(T, x_\varepsilon(\cdot)) = \sigma(x_\varepsilon(\cdot)) = \varphi_2(T, x_\varepsilon(\cdot))$ according to~\cref{boundary_condition}, we have $\Phi_\varepsilon (t_\varepsilon, x_\varepsilon(\cdot), \tau_\varepsilon, y_\varepsilon(\cdot)) \geq \Phi_\varepsilon (T, x_\varepsilon(\cdot), T, x_\varepsilon(\cdot)) = 0$, which yields
        \begin{equation} \label{proof_t_varepsilon-tau_varepsilon}
            (t_\varepsilon - \tau_\varepsilon)^2
            \leq K_1 \varepsilon^{\frac{3}{\alpha}},
        \end{equation}
        where
        \begin{displaymath}
            K_1
            \triangleq \max_{(t, x(\cdot)), (\tau, y(\cdot)) \in [0, T] \times X_k} \bigl( \varphi_1(t, x(\cdot)) - \varphi_2(\tau, y(\cdot)) \bigr)
            > 0.
        \end{displaymath}
        In particular, we obtain
        \begin{equation} \label{proof_t_varepsilon-tau_varepsilon_limit}
            |t_\varepsilon - \tau_\varepsilon|
            \to 0
            \text{ as } \varepsilon \to 0^+.
        \end{equation}

        Recall that the functionals $\varphi_1$ and $\varphi_2$ satisfy local Lipschitz continuity condition $(\rm L)$.
        Therefore, and in view of~\cref{lemma_Lipschitz_via_nu}, there is a number $K_2 > 0$ such that
        \begin{align}
            & |\varphi_1(t, x(\cdot)) - \varphi_1(t, y(\cdot))|
            + |\varphi_2(t, x(\cdot)) - \varphi_2(t, y(\cdot))| \label{bar_C_1} \\[0.2em]
            & \quad \leq K_2 \bigl( \nu_\varepsilon(t, x(\cdot), t, y(\cdot)) + C_1 \varepsilon^{\frac{q}{q - 1}} \bigr)^{\frac{1}{q}} \nonumber
        \end{align}
        for any $\varepsilon > 0$, any $t \in [0, T]$, and any $x(\cdot)$, $y(\cdot) \in X_k$.
        Let us put $K_3 \triangleq K_2 +  C_1^{\frac{q - 1}{q}}$ and show that, for all $\varepsilon > 0$,
        \begin{equation} \label{proof_nu_varepsilon}
            \bigl( \nu_\varepsilon (t_\varepsilon, x_\varepsilon(\cdot), \tau_\varepsilon, y_\varepsilon(\cdot))
            + C_1 \varepsilon^{\frac{q}{q - 1}} \bigr)^{\frac{q - 1}{q}}
            \leq K_3 \varepsilon.
        \end{equation}
        Fix $\varepsilon > 0$ and suppose that $\tau_\varepsilon \geq t_\varepsilon$ for definiteness.
        Note that $a(\cdot \mid t_\varepsilon, x_\varepsilon(\cdot)) \in X_k$ by~\cref{continuation_is_in_X_k} and, hence, $\Phi_\varepsilon (t_\varepsilon, x_\varepsilon(\cdot), \tau_\varepsilon, y_\varepsilon(\cdot)) \geq \Phi_\varepsilon \bigl( t_\varepsilon, x_\varepsilon(\cdot), \tau_\varepsilon, a(\cdot \mid t_\varepsilon, x_\varepsilon(\cdot)) \bigr)$.
        Then,
        \begin{align}
            & \varphi_2 \bigl( \tau_\varepsilon, a(\cdot \mid t_\varepsilon, x_\varepsilon(\cdot)) \bigr)
            - \varphi_2(\tau_\varepsilon, y_\varepsilon(\cdot)) \label{proof_nu_varepsilon_1} \\[0.1em]
            & \quad \geq \frac{\nu_\varepsilon (t_\varepsilon, x_\varepsilon(\cdot), \tau_\varepsilon, y_\varepsilon(\cdot))}{\varepsilon}
            - \frac{\nu_\varepsilon \bigl( t_\varepsilon, x_\varepsilon(\cdot), \tau_\varepsilon, a(\cdot \mid t_\varepsilon, x_\varepsilon(\cdot)) \bigr)}{\varepsilon}. \nonumber
        \end{align}
        On the other hand, owing to the choice of $K_2$ (see~\cref{bar_C_1}), we get
        \begin{align}
            & \varphi_2 \bigl( \tau_\varepsilon, a(\cdot \mid t_\varepsilon, x_\varepsilon(\cdot)) \bigr)
            - \varphi_2(\tau_\varepsilon, y_\varepsilon(\cdot)) \label{proof_nu_varepsilon_2} \\[0.1em]
            & \quad \leq K_2 \bigl( \nu_\varepsilon \bigl( \tau_\varepsilon, a(\cdot \mid t_\varepsilon, x_\varepsilon(\cdot)), \tau_\varepsilon, y_\varepsilon(\cdot) \bigr) + C_1 \varepsilon^{\frac{q}{q - 1}} \bigr)^{\frac{1}{q}}. \nonumber
        \end{align}
        In addition, due to~\cref{lemma_nu_properties_item_a,lemma_nu_properties_item_b}, we have
        \begin{equation} \label{proof_nu_varepsilon_3}
            \nu_\varepsilon \bigl( t_\varepsilon, x_\varepsilon(\cdot), \tau_\varepsilon, a(\cdot \mid t_\varepsilon, x_\varepsilon(\cdot)) \bigr)
            = \nu_\varepsilon (t_\varepsilon, x_\varepsilon(\cdot), t_\varepsilon, x_\varepsilon(\cdot))
            = 0
        \end{equation}
        and
        \begin{equation} \label{proof_nu_varepsilon_4}
            \nu_\varepsilon \bigl( \tau_\varepsilon, a(\cdot \mid t_\varepsilon, x_\varepsilon(\cdot)), \tau_\varepsilon, y_\varepsilon(\cdot) \bigr)
            = \nu_\varepsilon (t_\varepsilon, x_\varepsilon(\cdot), \tau_\varepsilon, y_\varepsilon(\cdot)).
        \end{equation}
        Thus, putting together relations~\cref{proof_nu_varepsilon_1,proof_nu_varepsilon_2,proof_nu_varepsilon_3,proof_nu_varepsilon_4}, we arrive at the inequality
        \begin{displaymath}
            \frac{\nu_\varepsilon(t_\varepsilon, x_\varepsilon(\cdot), \tau_\varepsilon, y_\varepsilon(\cdot))}{\varepsilon}
            \leq K_2 \bigl( \nu_\varepsilon (t_\varepsilon, x_\varepsilon(\cdot), \tau_\varepsilon, y_\varepsilon(\cdot))
            + C_1 \varepsilon^{\frac{q}{q - 1}} \bigr)^{\frac{1}{q}},
        \end{displaymath}
        which, together with the estimate
        \begin{displaymath}
            \frac{C_1 \varepsilon^{\frac{q}{q - 1}}}{\varepsilon}
            = C_1^{\frac{q - 1}{q}} \bigl( C_1 \varepsilon^{\frac{q}{q - 1}} \bigr)^{\frac{1}{q}}
            \leq C_1^{\frac{q - 1}{q}} \bigl( \nu_\varepsilon(t_\varepsilon, x_\varepsilon(\cdot), \tau_\varepsilon, y_\varepsilon(\cdot))
            + C_1 \varepsilon^{\frac{q}{q - 1}} \bigr)^{\frac{1}{q}},
        \end{displaymath}
        implies~\cref{proof_nu_varepsilon}.

        In particular, inequality~\cref{proof_nu_varepsilon} yields $\nu_\varepsilon (t_\varepsilon, x_\varepsilon(\cdot), \tau_\varepsilon, y_\varepsilon(\cdot)) \to 0$ as $\varepsilon \to 0^+$.
        Hence, in view of compactness of the set $X_k$, it follows from~\cref{lemma_convergene} that
        \begin{equation} \label{proof_a_t-a_tau_limit}
            \|a(\cdot \mid t_\varepsilon, x_\varepsilon(\cdot)) - a(\cdot \mid \tau_\varepsilon, y_\varepsilon(\cdot))\|_\infty
            \to 0
            \text{ as } \varepsilon \to 0^+.
        \end{equation}
        Moreover, by the Arzel\`{a}--Ascoli theorem, all functions from $X_k$ are equicontinuous, and, therefore, taking~\cref{proof_t_varepsilon-tau_varepsilon_limit} into account, we obtain
        \begin{equation} \label{proof_a_t-a_tau_limit_2}
            \|a(t_\varepsilon \mid \tau_\varepsilon, y_\varepsilon(\cdot)) - a(\tau_\varepsilon \mid \tau_\varepsilon, y_\varepsilon(\cdot))\|
            \to 0
            \text{ as } \varepsilon \to 0^+.
        \end{equation}
        According to~\cref{a}, we derive
        \begin{align*}
            & \|x_\varepsilon(t_\varepsilon) - y_\varepsilon(\tau_\varepsilon)\|
            = \|a(t_\varepsilon \mid t_\varepsilon, x_\varepsilon(\cdot)) - a(\tau_\varepsilon \mid \tau_\varepsilon, y_\varepsilon(\cdot))\| \\[0.15em]
            & \hphantom{\|x_\varepsilon(t_\varepsilon) - y_\varepsilon(\tau_\varepsilon)\|}
            \leq \|a(\cdot \mid t_\varepsilon, x_\varepsilon(\cdot)) - a(\cdot \mid \tau_\varepsilon, y_\varepsilon(\cdot))\|_\infty \\[0.15em]
            & \hphantom{\|x_\varepsilon(t_\varepsilon) - y_\varepsilon(\tau_\varepsilon)\|} \quad
            + \|a(t_\varepsilon \mid \tau_\varepsilon, y_\varepsilon(\cdot)) - a(\tau_\varepsilon \mid \tau_\varepsilon, y_\varepsilon(\cdot))\|
        \end{align*}
        for all $\varepsilon > 0$, wherefrom, applying~\cref{proof_a_t-a_tau_limit,proof_a_t-a_tau_limit_2}, we conclude that
        \begin{equation} \label{proof_x_varepsilon-y_varepsilon_limit}
            \|x_\varepsilon(t_\varepsilon) - y_\varepsilon(\tau_\varepsilon)\|
            \to 0
            \text{ as } \varepsilon \to 0^+.
        \end{equation}

        Further, since the functionals $\varphi_1$ and $\varphi_2$ are uniformly continuous on the compact set $[0, T] \times X_k$, there exists $\theta \in (0, T)$ such that
        \begin{equation} \label{proof_choice_of_theta}
            |\varphi_1(t, x(\cdot)) - \varphi_1(T, x(\cdot))|
            + |\varphi_2(t, x(\cdot)) - \varphi_2(T, x(\cdot))|
            \leq \frac{\varkappa}{8}
        \end{equation}
        for any $t \in [T - \theta, T]$ and any $x(\cdot) \in X_k$.
        Recall also that the functionals $\varphi_1$ and $\varphi_2$ are non-anticipative.
        Consequently, for $i \in \{1, 2\}$, we have
        \begin{displaymath}
            \varphi_i(t_\varepsilon, x_\varepsilon(\cdot))
            = \varphi_i \bigl( t_\varepsilon, a(\cdot \mid t_\varepsilon, x_\varepsilon(\cdot)) \bigr),
            \quad
            \varphi_i(\tau_\varepsilon, y_\varepsilon(\cdot))
            = \varphi_i \bigl( \tau_\varepsilon, a(\cdot \mid \tau_\varepsilon, y_\varepsilon(\cdot)) \bigr)
        \end{displaymath}
        for all $\varepsilon > 0$, and, hence, relations~\cref{proof_t_varepsilon-tau_varepsilon_limit,proof_a_t-a_tau_limit} imply that
        \begin{displaymath}
            | \varphi_i(t_\varepsilon, x_\varepsilon(\cdot))
            - \varphi_i(\tau_\varepsilon, y_\varepsilon(\cdot))|
            \to 0
            \text{ as } \varepsilon \to 0^+.
        \end{displaymath}
        Thus, there is a number $\varepsilon_\ast > 0$ such that, for any $\varepsilon \in (0, \varepsilon_\ast]$,
        \begin{equation} \label{proof_choice_varepsilon_1}
            | \varphi_1(t_\varepsilon, x_\varepsilon(\cdot))
            - \varphi_1(\tau_\varepsilon, y_\varepsilon(\cdot))|
            + | \varphi_2(t_\varepsilon, x_\varepsilon(\cdot))
            - \varphi_2(\tau_\varepsilon, y_\varepsilon(\cdot))|
            \leq \frac{\varkappa}{4}.
        \end{equation}
        Let us show that, for every $\varepsilon \in (0, \varepsilon_\ast]$, the inclusions below are fulfilled:
        \begin{displaymath}
            t_\varepsilon, \tau_\varepsilon
            \in [0, T - \theta).
        \end{displaymath}
        In accordance with~\cref{proof_max}, we obtain
        \begin{align*}
            & \varphi_1(t_\varepsilon, x_\varepsilon(\cdot)) - \varphi_2(\tau_\varepsilon, y_\varepsilon(\cdot))
            \geq \Phi_\varepsilon (t_\varepsilon, x_\varepsilon(\cdot), \tau_\varepsilon, y_\varepsilon(\cdot)) \\[0.15em]
            & \hphantom{\varphi_1(t_\varepsilon, x_\varepsilon(\cdot)) - \varphi_2(\tau_\varepsilon, y_\varepsilon(\cdot))}
            \geq \Phi_\varepsilon (t, x(\cdot), t, x(\cdot)) \\[0.15em]
            & \hphantom{\varphi_1(t_\varepsilon, x_\varepsilon(\cdot)) - \varphi_2(\tau_\varepsilon, y_\varepsilon(\cdot))}
            = \varphi_1(t, x(\cdot)) - \varphi_2(t, x(\cdot)) - 2 (T - t) \zeta
        \end{align*}
        for all $(t, x(\cdot)) \in [0, T] \times X_k$, wherefrom, by the definition of $\varkappa$ and $\zeta$ (see~\cref{proof_choice_of_varkappa,proof_choice_of_zeta}), it follows that
        \begin{displaymath}
            \varphi_1(t_\varepsilon, x_\varepsilon(\cdot)) - \varphi_2(\tau_\varepsilon, y_\varepsilon(\cdot))
            \geq \varkappa - 2 T \zeta
            = \frac{\varkappa}{2}.
        \end{displaymath}
        Then, taking into account that $\varphi_1(T, x_\varepsilon(\cdot)) = \sigma(x_\varepsilon(\cdot)) = \varphi_2(T, x_\varepsilon(\cdot))$ due to~\cref{boundary_condition}, we derive
        \begin{align*}
            & |\varphi_1(t_\varepsilon, x_\varepsilon(\cdot)) - \varphi_1(T, x_\varepsilon(\cdot))|
            + |\varphi_2(T, x_\varepsilon(\cdot)) - \varphi_2(t_\varepsilon, x_\varepsilon(\cdot))| \\[0.2em]
            & \quad + |\varphi_2(t_\varepsilon, x_\varepsilon(\cdot)) - \varphi_2(\tau_\varepsilon, y_\varepsilon(\cdot))|
            \geq \frac{\varkappa}{2}.
        \end{align*}
        Therefore, in view of the choice of $\varepsilon_\ast$ (see~\cref{proof_choice_varepsilon_1}), the estimate
        \begin{displaymath}
            |\varphi_1(t_\varepsilon, x_\varepsilon(\cdot)) - \varphi_1(T, x_\varepsilon(\cdot))|
            + |\varphi_2(T, x_\varepsilon(\cdot)) - \varphi_2(t_\varepsilon, x_\varepsilon(\cdot))|
            \geq \frac{\varkappa}{4}
        \end{displaymath}
        holds, which, owing to the choice of $\theta$ (see~\cref{proof_choice_of_theta}), implies that $t_\varepsilon < T - \theta$.
        The inequality $\tau_\varepsilon < T - \theta$ can be verified in a similar way.

        Now, for every $\varepsilon \in (0, \varepsilon_\ast]$, consider a functional $\psi_1 \colon [0, T] \times AC^\alpha \to \mathbb{R}$ given by
        \begin{displaymath}
            \psi_1(t, x(\cdot))
            \triangleq \varphi_2(\tau_\varepsilon, y_\varepsilon(\cdot)) + (2 T - t - \tau_\varepsilon) \zeta
            + \frac{(t - \tau_\varepsilon)^2}{\varepsilon^{\frac{3}{\alpha}}}
            + \frac{\mu_\varepsilon^{(\tau_\varepsilon, y_\varepsilon(\cdot))}(t, x(\cdot))}{\varepsilon}
        \end{displaymath}
        for all $(t, x(\cdot)) \in [0, T] \times AC^\alpha$, where the functional $\mu_\varepsilon^{(\tau_\varepsilon, y_\varepsilon(\cdot))}$ is defined according to~\cref{psi_varepsilon}.
        Applying~\cref{lemma_ci_smoothness}, we obtain that the functional $\psi_1$ is $ci$-smooth of the order $\alpha$ and its $ci$-derivatives of the order $\alpha$ are as follows:
        \begin{displaymath}
            \partial_t^\alpha \psi_1(t, x(\cdot))
            = - \zeta + \frac{2 (t - \tau_\varepsilon)}{\varepsilon^{\frac{3}{\alpha}}},
            \quad \nabla^\alpha \psi_1(t, x(\cdot))
            = \frac{\nabla^\alpha \mu_\varepsilon^{(\tau_\varepsilon, y_\varepsilon(\cdot))} (t, x(\cdot))}{\varepsilon}
        \end{displaymath}
        for all $(t, x(\cdot)) \in [0, T) \times AC^\alpha$.
        In addition, by construction, we have
        \begin{align*}
            & \varphi_1(t, x(\cdot)) - \psi_1(t, x(\cdot))
            = \Phi_\varepsilon (t, x(\cdot), \tau_\varepsilon, y_\varepsilon(\cdot)) \\[0.15em]
            & \hphantom{\varphi_1(t, x(\cdot)) - \psi_1(t, x(\cdot))}
            \leq \Phi_\varepsilon (t_\varepsilon, x_\varepsilon(\cdot), \tau_\varepsilon, y_\varepsilon(\cdot)) \\[0.15em]
            & \hphantom{\varphi_1(t, x(\cdot)) - \psi_1(t, x(\cdot))}
            = \varphi_1(t_\varepsilon, x_\varepsilon(\cdot)) - \psi_1(t_\varepsilon, x_\varepsilon(\cdot))
        \end{align*}
        for all $(t, x(\cdot)) \in [0, T] \times X_k$.
        Hence, since the functional $\varphi_1$ possesses property $(\rm V_-)$ and the inequality $t_\varepsilon < T$ is valid, we conclude that
        \begin{equation} \label{proof_HJ_inequality_x}
            - \zeta + \frac{2 (t_\varepsilon - \tau_\varepsilon)}{\varepsilon^{\frac{3}{\alpha}}}
            + H \biggl( t_\varepsilon, x_\varepsilon(t_\varepsilon),
            \frac{\nabla^\alpha \mu_\varepsilon^{(\tau_\varepsilon, y_\varepsilon(\cdot))} (t_\varepsilon, x_\varepsilon(\cdot))}{\varepsilon} \biggr)
            \geq 0.
        \end{equation}
        On the other hand, define a functional $\psi_2 \colon [0, T] \times AC^\alpha \to \mathbb{R}$ by
        \begin{displaymath}
            \psi_2(\tau, y(\cdot))
            \triangleq \varphi_1(t_\varepsilon, x_\varepsilon(\cdot)) - (2 T - t_\varepsilon - \tau) \zeta
            - \frac{(t_\varepsilon - \tau)^2}{\varepsilon^{\frac{3}{\alpha}}}
            - \frac{\mu_\varepsilon^{(t_\varepsilon, x_\varepsilon(\cdot))} (\tau, y(\cdot))}{\varepsilon}
        \end{displaymath}
        for all $(\tau, y(\cdot)) \in [0, T] \times AC^\alpha$.
        The functional $\psi_2$ is $ci$-smooth of the order $\alpha$ and
        \begin{displaymath}
            \partial_t^\alpha \psi_2(\tau, y(\cdot))
            = \zeta + \frac{2 (t_\varepsilon - \tau)}{\varepsilon^{\frac{3}{\alpha}}},
            \quad \nabla^\alpha \psi_2(\tau, y(\cdot))
            = - \frac{\nabla^\alpha \mu_\varepsilon^{(t_\varepsilon, x_\varepsilon(\cdot))} (\tau, y(\cdot))}{\varepsilon}
        \end{displaymath}
        for all $(\tau, y(\cdot)) \in [0, T) \times AC^\alpha$.
        Moreover, taking the second equality in~\cref{lemma_nu_properties_item_a} into account, we get
        \begin{align*}
            & \psi_2(\tau, y(\cdot)) - \varphi_2(\tau, y(\cdot))
            = \Phi_\varepsilon (t_\varepsilon, x_\varepsilon(\cdot), \tau, y(\cdot)) \\[0.15em]
            & \hphantom{\psi_2(\tau, y(\cdot)) - \varphi_2(\tau, y(\cdot))}
            \leq \Phi_\varepsilon (t_\varepsilon, x_\varepsilon(\cdot), \tau_\varepsilon, y_\varepsilon(\cdot)) \\[0.15em]
            & \hphantom{\psi_2(\tau, y(\cdot)) - \varphi_2(\tau, y(\cdot))}
            = \psi_2(\tau_\varepsilon, y_\varepsilon(\cdot)) - \varphi_2(\tau_\varepsilon, y_\varepsilon(\cdot))
        \end{align*}
        for all $(\tau, y(\cdot)) \in [0, T] \times X_k$.
        Therefore, due to property $(\rm V_+)$ of the functional $\varphi_2$ and the inequality $\tau_\varepsilon < T$, we have
        \begin{equation} \label{proof_HJ_inequality_y}
            \zeta + \frac{2 (t_\varepsilon - \tau_\varepsilon)}{\varepsilon^{\frac{3}{\alpha}}}
            + H \biggl( \tau_\varepsilon, y_\varepsilon(\tau_\varepsilon),
            - \frac{\nabla^\alpha \mu_\varepsilon^{(t_\varepsilon, x_\varepsilon(\cdot))} (\tau_\varepsilon, y_\varepsilon(\cdot))}{\varepsilon} \biggr)
            \leq 0.
        \end{equation}
        As a result, we derive from~\cref{proof_HJ_inequality_x,proof_HJ_inequality_y} that, for every $\varepsilon \in (0, \varepsilon_\ast]$,
        \begin{align}
            & 2 \zeta
            \leq H \biggl( t_\varepsilon, x_\varepsilon(t_\varepsilon),
            \frac{\nabla^\alpha \mu_\varepsilon^{(\tau_\varepsilon, y_\varepsilon(\cdot))} (t_\varepsilon, x_\varepsilon(\cdot))}{\varepsilon} \biggr)
            \label{proof_final_difference} \\[0.1em]
            & \hphantom{2 \zeta} \quad
            - H \biggl( \tau_\varepsilon, y_\varepsilon(\tau_\varepsilon),
            - \frac{\nabla^\alpha \mu_\varepsilon^{(t_\varepsilon, x_\varepsilon(\cdot))} (\tau_\varepsilon, y_\varepsilon(\cdot))}{\varepsilon} \biggr).
            \nonumber
        \end{align}

        Let $R > 0$ be such that $\|x(\cdot)\|_\infty \leq R$ for all $x(\cdot) \in X_k$ and let $C_3$ and $C_4$ be the numbers that, according to~\cref{lemma_gradient_properties}, correspond to the chosen $\theta$ (see~\cref{proof_choice_of_theta}).
        Then, for every $\varepsilon \in (0, \varepsilon_\ast]$, relying on~\cref{difference_of_gradients,proof_t_varepsilon-tau_varepsilon}, we obtain
        \begin{align*}
            & \biggl\| \frac{\nabla^\alpha \mu_\varepsilon^{(\tau_\varepsilon, y_\varepsilon(\cdot))} (t_\varepsilon, x_\varepsilon(\cdot))}{\varepsilon}
            + \frac{\nabla^\alpha \mu_\varepsilon^{(t_\varepsilon, x_\varepsilon(\cdot))} (\tau_\varepsilon, y_\varepsilon(\cdot))}{\varepsilon} \biggr\| \\[0.1em]
            & \quad \leq \frac{C_4 \bigl( \varepsilon^{\frac{2}{q - 1}} + 4 R^2 \bigr)^{\frac{q - 1}{2}} |t_\varepsilon - \tau_\varepsilon|^\alpha}{\varepsilon}
            \leq K_4 \varepsilon^{\frac{1}{2}},
        \end{align*}
        where $K_4 \triangleq C_4 \bigl( \varepsilon_\ast^{\frac{2}{q - 1}} + 4 R^2 \bigr)^{\frac{q - 1}{2}} K_1^{\frac{\alpha}{2}}$, and, consequently, in view of property (jj) of the Hamiltonian $H$, we have
        \begin{align}
            & \biggl| H \biggl( \tau_\varepsilon, y_\varepsilon(\tau_\varepsilon),
            \frac{\nabla^\alpha \mu_\varepsilon^{(\tau_\varepsilon, y_\varepsilon(\cdot))} (t_\varepsilon, x_\varepsilon(\cdot))}{\varepsilon} \biggr)
            \label{contradiction_first} \\[0.2em]
            & \quad - H \biggl( \tau_\varepsilon, y_\varepsilon(\tau_\varepsilon),
            - \frac{\nabla^\alpha \mu_\varepsilon^{(t_\varepsilon, x_\varepsilon(\cdot))} (\tau_\varepsilon, y_\varepsilon(\cdot))}{\varepsilon} \biggr) \biggr|
            \leq c_\ast (1 + R) K_4 \varepsilon^{\frac{1}{2}}.
            \nonumber
        \end{align}
        Further, applying~\cref{gradient_bound,proof_nu_varepsilon}, we get
        \begin{displaymath}
            \biggl\| \frac{\nabla^\alpha \mu_\varepsilon^{(\tau_\varepsilon, y_\varepsilon(\cdot))}(t_\varepsilon, x_\varepsilon(\cdot))}{\varepsilon} \biggr\|
            \leq \frac{C_3 \bigl( \nu_\varepsilon(t_\varepsilon, x_\varepsilon(\cdot), \tau_\varepsilon, y_\varepsilon(\cdot)) + C_1 \varepsilon^{\frac{q}{q - 1}} \bigr)^{\frac{q - 1}{q}}}{\varepsilon}
            \leq K_5
        \end{displaymath}
        for all $\varepsilon \in (0, \varepsilon_\ast]$, where $K_5 \triangleq C_3 K_3$.
        Hence, since the Hamiltonian $H$ is uniformly continuous on the compact set $[0, T] \times B_R \times B_{K_5}$ by property (j) and relations~\cref{proof_t_varepsilon-tau_varepsilon_limit,proof_x_varepsilon-y_varepsilon_limit} hold, we derive
        \begin{displaymath}
            \biggl| H \biggl( t_\varepsilon, x_\varepsilon(t_\varepsilon),
            \frac{\nabla^\alpha \mu_\varepsilon^{(\tau_\varepsilon, y_\varepsilon(\cdot))} (t_\varepsilon, x_\varepsilon(\cdot))}{\varepsilon} \biggr)
            - H \biggl( \tau_\varepsilon, y_\varepsilon(\tau_\varepsilon),
            \frac{\nabla^\alpha \mu_\varepsilon^{(\tau_\varepsilon, y_\varepsilon(\cdot))} (t_\varepsilon, x_\varepsilon(\cdot))}{\varepsilon} \biggr) \biggr|
            \to 0
        \end{displaymath}
        as $\varepsilon \to 0^+$.
        From this relation and estimate~\cref{contradiction_first}, it follows that the right-hand side of inequality~\cref{proof_final_difference} tends to $0$ as $\varepsilon \to 0^+$.
        Thus, we conclude that $2 \zeta \leq 0$ and obtain a contradiction to~\cref{proof_choice_of_zeta}.
        The proof is complete.

\appendix

\section{Proofs of~\cref{subsection_Lemmas}}
\label{Appendix}

    \begin{proof}[Proof of~\cref{lemma_nu_is_continuous}]
        Fix $\varepsilon > 0$.
        Directly from the definition of the functional $\nu_\varepsilon$, it follows that this functional is non-negative and that equalities~\cref{lemma_nu_properties_item_a} are valid, while equality~\cref{lemma_nu_properties_item_b} is a consequence of relation~\cref{semigroup_property}.
        Thus, it remains to verify continuity of the functional $\nu_\varepsilon$.
        Let $(t_0, x_0(\cdot))$, $(\tau_0, y_0(\cdot)) \in [0, T] \times AC^\alpha$ and let sequences $\{(t_i, x_i(\cdot))\}_{i \in \mathbb{N}}$, $\{(\tau_i, y_i(\cdot))\}_{i \in \mathbb{N}} \subset [0, T] \times AC^\alpha$ be such that, as  $i \to \infty$,
        \begin{displaymath}
            |t_i - t_0| + \|x_i(\cdot) - x_0(\cdot)\|_\infty
            \to 0,
            \quad |\tau_i - \tau_0| + \|y_i(\cdot) - y_0(\cdot)\|_\infty \to 0.
        \end{displaymath}
        For every $i \in \mathbb{N} \cup \{0\}$, consider the function
        \begin{displaymath}
            b_i(\xi)
            \triangleq \bigl( \varepsilon^{\frac{2}{q - 1}} + \|a(\xi \mid t_i, x_i(\cdot)) - a(\xi \mid \tau_i, y_i(\cdot))\|^2 \bigr)^{\frac{q}{2}},
            \quad \xi \in [0, T].
        \end{displaymath}
        Then, we have
        \begin{displaymath}
            \nu_\varepsilon(t_i, x_i(\cdot), \tau_i, y_i(\cdot))
            = b_i(T) + \int_{0}^{T} \frac{b_i(\xi)}{(T - \xi)^{(1 - \alpha - \beta) q}} \, \rd \xi - C_1 \varepsilon^{\frac{q}{q - 1}}
            \quad \forall i \in \mathbb{N} \cup \{0\}.
        \end{displaymath}
        Hence, since $\|b_i(\cdot) - b_0(\cdot)\|_\infty \to 0$ as $i \to \infty$ by continuity of mapping~\cref{a_mapping}, we find that $\nu_\varepsilon (t_i, x_i(\cdot), \tau_i, y_i(\cdot)) \to \nu_\varepsilon (t_0, x_0(\cdot), \tau_0, y_0(\cdot))$ as $i \to \infty$ and complete the proof.
    \end{proof}

    \begin{proof}[Proof of~\cref{lemma_Lipschitz_via_nu}]
        Note that $\beta q / (q - 1) < 1$ by the condition $\beta < \alpha / 2$ and put
        \begin{equation} \label{proof_f_C_2}
            C_2
            \triangleq 1 + \frac{T^{\frac{q - 1}{q} - \beta}}{\bigl( 1 - \frac{\beta q}{q - 1} \bigr)^{\frac{q - 1}{q}}}.
        \end{equation}
        For any $\varepsilon > 0$ and any $(t, x(\cdot))$, $(\tau, y(\cdot)) \in [0, T] \times AC^\alpha$, according to~\cref{nu}, we get
        \begin{align}
            & \|a(T \mid t, x(\cdot)) - a(T \mid \tau, y(\cdot))\| \label{proof_f_1} \\[0.1em]
            & \quad \leq \Bigl( \bigl( \varepsilon^{\frac{2}{q - 1}} + \|a(T \mid t, x(\cdot)) - a(T \mid \tau, y(\cdot))\|^2 \bigr)^{\frac{q}{2}} \Bigr)^{\frac{1}{q}} \nonumber \\[0.1em]
            & \quad \leq \bigl( \nu_\varepsilon(t, x(\cdot), \tau, y(\cdot)) + C_1 \varepsilon^{\frac{q}{q - 1}} \bigr)^{\frac{1}{q}}
            \nonumber
        \end{align}
        and, by H\"{o}lder's inequality,
        \begin{align}
            & \int_{0}^{T} \frac{\|a(\xi \mid t, x(\cdot)) - a(\xi \mid \tau, y(\cdot))\|}{(T - \xi)^{1 - \alpha}} \, \rd \xi
            \label{proof_f_2} \\[0.1em]
            & \hspace*{-0.5em} \leq \int_{0}^{T} \frac{\bigl( \varepsilon^{\frac{2}{q - 1}} + \|a(\xi \mid t, x(\cdot)) - a(\xi \mid \tau, y(\cdot))\|^2 \bigr)^{\frac{1}{2}}}
            {(T - \xi)^{1 - \alpha}} \, \rd \xi
            \nonumber \\[0.1em]
            & \hspace*{-0.5em} \leq \biggl( \int_{0}^{T} \frac{\rd \xi}{(T - \xi)^{\frac{\beta q}{q - 1}}} \biggr)^{\frac{q - 1}{q}}
            \biggl( \int_{0}^{T} \frac{\bigl( \varepsilon^{\frac{2}{q - 1}} + \|a(\xi \mid t, x(\cdot)) - a(\xi \mid \tau, y(\cdot))\|^2 \bigr)^{\frac{q}{2}}}
            {(T - \xi)^{(1 - \alpha - \beta) q}} \, \rd \xi \biggr)^{\frac{1}{q}}
            \nonumber \\[0.1em]
            & \hspace*{-0.5em} \leq \frac{T^{\frac{q - 1}{q} - \beta}}{\bigl( 1 - \frac{\beta q}{q - 1} \bigr)^{\frac{q - 1}{q}}}
            \bigl( \nu_\varepsilon(t, x(\cdot), \tau, y(\cdot)) + C_1 \varepsilon^{\frac{q}{q - 1}} \bigr)^{\frac{1}{q}}.
            \nonumber
        \end{align}
        From~\cref{proof_f_1,proof_f_2}, we derive estimate~\cref{Lipschitz_via_nu_main} with $C_2$ given by~\cref{proof_f_C_2}.
    \end{proof}

    \begin{proof}[Proof of~\cref{lemma_convergene}]
        Fix a compact set $X \subset AC^\alpha$ and, for every $\varepsilon > 0$, take $(t_\varepsilon, x_\varepsilon(\cdot))$, $(\tau_\varepsilon, y_\varepsilon(\cdot)) \in [0, T] \times X$ such that $\nu_\varepsilon(t_\varepsilon, x_\varepsilon(\cdot), \tau_\varepsilon, y_\varepsilon(\cdot)) \to 0$ as $\varepsilon \to 0^+$.
        Arguing by contradiction, assume that relation~\cref{convergence} does not hold.
        Then, there exists $\varkappa > 0$ such that, for any $i \in \mathbb{N}$, we can choose $\varepsilon_i \in (0, 1 / i]$ from the condition
        \begin{equation} \label{supposition}
            \|a(\cdot \mid t_{\varepsilon_i}, x_{\varepsilon_i}(\cdot)) - a(\cdot \mid \tau_{\varepsilon_i}, y_{\varepsilon_i}(\cdot))\|_\infty
            \geq \varkappa.
        \end{equation}
        In view of compactness of the set $[0, T] \times X$, we can suppose that the sequences $\{(t_{\varepsilon_i}, x_{\varepsilon_i}(\cdot))\}_{i \in \mathbb{N}}$, $\{(\tau_{\varepsilon_i}, y_{\varepsilon_i}(\cdot))\}_{i \in \mathbb{N}}$ converge respectively to some points $(t_0, x_0(\cdot))$, $(\tau_0, y_0(\cdot)) \in [0, T] \times X$.
        Hence, due to continuity of mapping~\cref{a_mapping}, we have
        \begin{displaymath}
            \|a(\cdot \mid t_{\varepsilon_i}, x_{\varepsilon_i}(\cdot)) - a(\cdot \mid t_0, x_0(\cdot))\|_\infty
            \to 0,
            \quad \|a(\cdot \mid \tau_{\varepsilon_i}, y_{\varepsilon_i}(\cdot)) - a(\cdot \mid \tau_0, y_0(\cdot))\|_\infty
            \to 0
        \end{displaymath}
        as $i \to \infty$.
        Thus, on the one hand, it follows from~\cref{supposition} that
        \begin{equation} \label{supposition_2}
            \|a(\cdot \mid t_0, x_0(\cdot)) - a(\cdot \mid \tau_0, y_0(\cdot))\|_\infty
            \geq \varkappa.
        \end{equation}
        But, on the other hand, in accordance with~\cref{nu}, we get, for all $i \in \mathbb{N}$,
        \begin{align*}
            & \int_{0}^{T} \frac{\|a(\xi \mid t_{\varepsilon_i}, x_{\varepsilon_i}(\cdot)) - a(\xi \mid \tau_{\varepsilon_i}, y_{\varepsilon_i}(\cdot))\|^q}
            {(T - \xi)^{(1 - \alpha - \beta) q}} \, \rd \xi \\[0.1em]
            & \quad \leq \int_{0}^{T} \frac{\bigl( \varepsilon_i^{\frac{2}{q - 1}}
            + \|a(\xi \mid t_{\varepsilon_i}, x_{\varepsilon_i}(\cdot)) - a(\xi \mid \tau_{\varepsilon_i}, y_{\varepsilon_i}(\cdot))\|^2 \bigr)^{\frac{q}{2}}}
            {(T - \xi)^{(1 - \alpha - \beta) q}} \, \rd \xi \\[0.2em]
            & \quad \leq \nu_{\varepsilon_i}(t_{\varepsilon_i}, x_{\varepsilon_i}(\cdot), \tau_{\varepsilon_i}, y_{\varepsilon_i}(\cdot))
            + C_1 \varepsilon_i^{\frac{q}{q - 1}}.
        \end{align*}
        Therefore, since $\varepsilon_i \to 0^+$ and $\nu_{\varepsilon_i}(t_{\varepsilon_i}, x_{\varepsilon_i}(\cdot), \tau_{\varepsilon_i}, y_{\varepsilon_i}(\cdot)) \to 0$ as $i \to \infty$, we obtain
        \begin{displaymath}
            \int_{0}^{T} \frac{\|a(\xi \mid t_{\varepsilon_i}, x_{\varepsilon_i}(\cdot)) - a(\xi \mid \tau_{\varepsilon_i}, y_{\varepsilon_i}(\cdot))\|^q}
            {(T - \xi)^{(1 - \alpha - \beta) q}} \, \rd \xi
            \to 0
            \text{ as } i \to \infty.
        \end{displaymath}
        Consequently, it holds that
        \begin{displaymath}
            \int_{0}^{T} \frac{\|a(\xi \mid t_0, x_0(\cdot)) - a(\xi \mid \tau_0, y_0(\cdot))\|^q}
            {(T - \xi)^{(1 - \alpha - \beta) q}} \, \rd \xi
            = 0,
        \end{displaymath}
        wherefrom, owing to continuity of the functions $a(\cdot \mid t_0, x_0(\cdot))$ and $a(\cdot \mid \tau_0, y_0(\cdot))$, we derive the equality $\|a(\cdot \mid t_0, x_0(\cdot)) - a(\cdot \mid \tau_0, y_0(\cdot))\|_\infty = 0$, contradicting~\cref{supposition_2}.
    \end{proof}

    Before proceeding with the proofs of~\cref{lemma_ci_smoothness,lemma_gradient_properties}, we recall that, for any $\gamma \in (0, 1)$ and any $t \in [0, T)$, the equality below is valid (see, e.g.,~\cite[Example~2.1]{Diethelm_2010}):
    \begin{equation} \label{Beta}
        \int_{t}^{T} \frac{\rd \xi}{(\xi - t)^{1 - \gamma} (T - \xi)^{(1 - \alpha - \beta) q}}
        = B(\gamma, 1 - (1 - \alpha - \beta) q) (T - t)^{\gamma - (1 - \alpha - \beta) q},
    \end{equation}
    where $B$ is the beta-function.
    Hence, in particular, for every $\theta \in (0, T)$, there exists a number $A_1 > 0$ such that
    \begin{align}
        & \bigg| \int_{t}^{T} \frac{\rd \xi}{(\xi - t)^{1 - \alpha} (T - \xi)^{(1 - \alpha - \beta) q}}
        - \int_{\tau}^{T} \frac{\rd \xi}{(\xi - \tau)^{1 - \alpha} (T - \xi)^{(1 - \alpha - \beta) q}} \bigg| \label{tilde_C_1} \\[0.5em]
        & \quad \leq A_1 |t - \tau| \nonumber
    \end{align}
    for all $t$, $\tau \in [0, T - \theta]$.

    \begin{proof}[Proof of~\cref{lemma_ci_smoothness}]
        Let $\varepsilon > 0$ and $(\tau_\ast, y_\ast(\cdot)) \in [0, T] \times AC^\alpha$ be fixed.
        Note that continuity of the functional $\mu_\varepsilon^{(\tau_\ast, y_\ast(\cdot))}$ follows directly from continuity of the functional $\nu_\varepsilon$ (see~\cref{lemma_nu_is_continuous}).

        Let us take $(t, x(\cdot)) \in [0, T) \times AC^\alpha$ and show that the functional $\mu_\varepsilon^{(\tau_\ast, y_\ast(\cdot))}$ is $ci$-differentiable of the order $\alpha$ at $(t, x(\cdot))$ and that equalities~\cref{lemma_nu_properties_item_c_derivative_t,lemma_nu_properties_item_c_derivative_x} hold.
        To this end, we need to consider a function $y(\cdot) \in Y(t, x(\cdot))$ (see~\cref{Y}) and verify that
        \begin{align}
            & \mu_\varepsilon^{(\tau_\ast, y_\ast(\cdot))} (\tau, y(\cdot)) - \mu_\varepsilon^{(\tau_\ast, y_\ast(\cdot))} (t, x(\cdot))
            \label{lemma_ci_smoothness_definition_of_ci_differentiability} \\[0.25em]
            & - \frac{q \langle a(T \mid t, x(\cdot)) - a_\ast(T), z(\tau) \rangle}
            {\Gamma(\alpha) \bigl( \varepsilon^{\frac{2}{q - 1}} + \|a(T \mid t, x(\cdot)) - a_\ast(T)\|^2 \bigr)^{1 - \frac{q}{2}}
            (T - t)^{1 - \alpha}}
            \nonumber \\[0.25em]
            & - \int_{t}^{T} \frac{q \langle a(\xi \mid t, x(\cdot)) - a_\ast(\xi), z(\tau) \rangle}
            {\Gamma(\alpha) \bigl( \varepsilon^{\frac{2}{q - 1}} + \|a(\xi \mid t, x(\cdot)) - a_\ast(\xi)\|^2 \bigr)^{1 - \frac{q}{2}}
            (\xi - t)^{1 - \alpha} (T - \xi)^{(1 - \alpha - \beta) q}} \, \rd \xi
            \nonumber \\[0.3em]
            & \quad
            = o(\tau - t)
            \nonumber
        \end{align}
        for all $\tau \in (t, T)$, where we denote $a_\ast(\cdot) \triangleq a(\cdot \mid \tau_\ast, y_\ast(\cdot))$ and
        \begin{displaymath}
            z(\tau)
            \triangleq \int_{t}^{\tau} (^C D^\alpha y)(\xi) \, \rd \xi,
            \quad \tau \in [t, T].
        \end{displaymath}
        Since, for any $\tau \in [t, T]$ and any $\xi \in [0, t]$, in accordance with~\cref{a}, it holds that $a(\xi \mid \tau, y(\cdot)) = y(\xi) = x(\xi) = a(\xi \mid t, x(\cdot))$, we have
        \begin{align*}
            & \mu_\varepsilon^{(\tau_\ast, y_\ast(\cdot))} (\tau, y(\cdot)) - \mu_\varepsilon^{(\tau_\ast, y_\ast(\cdot))} (t, x(\cdot)) \\[0.25em]
            & \ \ = \bigl( \varepsilon^{\frac{2}{q - 1}} + \|a(T \mid \tau, y(\cdot)) - a_\ast(T)\|^2 \bigr)^{\frac{q}{2}}
            - \bigl( \varepsilon^{\frac{2}{q - 1}} + \|a(T \mid t, x(\cdot)) - a_\ast(T)\|^2 \bigr)^{\frac{q}{2}} \\[0.25em]
            & \ \ \ \
            + \int_{t}^{T} \frac{\bigl( \varepsilon^{\frac{2}{q - 1}} + \|a(\xi \mid \tau, y(\cdot)) - a_\ast(\xi)\|^2 \bigr)^{\frac{q}{2}}
            - \bigl( \varepsilon^{\frac{2}{q - 1}} + \|a(\xi \mid t, x(\cdot)) - a_\ast(\xi)\|^2 \bigr)^{\frac{q}{2}}}
            {(T - \xi)^{(1 - \alpha - \beta) q}} \, \rd \xi
        \end{align*}
        for all $\tau \in [t, T]$.
        Consequently, in order to obtain~\cref{lemma_ci_smoothness_definition_of_ci_differentiability}, it suffices to fix $\theta \in (0, T - t)$ and prove that there are numbers $A_2 > 0$ and $A_3 > 0$ such that, for any $\tau \in (t, T - \theta]$,
        \begin{align}
            & \biggl| \bigl( \varepsilon^{\frac{2}{q - 1}} + \|a(T \mid \tau, y(\cdot)) - a_\ast(T)\|^2 \bigr)^{\frac{q}{2}}
            - \bigl( \varepsilon^{\frac{2}{q - 1}} + \|a(T \mid t, x(\cdot)) - a_\ast(T)\|^2 \bigr)^{\frac{q}{2}}
            \label{lemma_proof_ci_differentiability_first_term} \\[0.25em]
            & \quad - \frac{q \langle a(T \mid t, x(\cdot)) - a_\ast(T), z(\tau) \rangle}
            {\Gamma(\alpha) \bigl( \varepsilon^{\frac{2}{q - 1}} + \|a(T \mid t, x(\cdot)) - a_\ast(T)\|^2 \bigr)^{1 - \frac{q}{2}} (T - t)^{1 - \alpha}} \biggr|
            \leq A_2 (t - \tau)^2
            \nonumber
        \end{align}
        and
        \begin{align}
            & \biggl| \int_{t}^{T} \frac{\bigl( \varepsilon^{\frac{2}{q - 1}} + \|a(\xi \mid \tau, y(\cdot)) - a_\ast(\xi)\|^2 \bigr)^{\frac{q}{2}}
            - \bigl( \varepsilon^{\frac{2}{q - 1}} + \|a(\xi \mid t, x(\cdot)) - a_\ast(\xi)\|^2 \bigr)^{\frac{q}{2}}}
            {(T - \xi)^{(1 - \alpha - \beta) q}} \, \rd \xi \label{lemma_proof_ci_differentiability_second_term} \\[0.3em]
            & \quad - \int_{t}^{T} \frac{q \langle a(\xi \mid t, x(\cdot)) - a_\ast(\xi), z(\tau) \rangle}
            {\Gamma(\alpha) \bigl( \varepsilon^{\frac{2}{q - 1}} + \|a(\xi \mid t, x(\cdot)) - a_\ast(\xi)\|^2 \bigr)^{1 - \frac{q}{2}}
            (\xi - t)^{1 - \alpha} (T - \xi)^{(1 - \alpha - \beta) q}} \, \rd \xi \biggr|
            \nonumber \\[0.4em]
            & \qquad
            \leq A_3 (\tau - t)^{1 + \alpha}.
            \nonumber
        \end{align}

        Let us derive some preliminary estimates.
        Choose $R > 0$ satisfying the condition $\|y(\cdot)\|_\infty + \|a_\ast(\cdot)\|_\infty \leq R$.
        Then, by Taylor's expansion, there exists $A_4 > 0$ such that
        \begin{displaymath}
            \biggl| \bigl( \varepsilon^{\frac{2}{q - 1}} + \|r\|^2 \bigr)^{\frac{q}{2}} - \bigl( \varepsilon^{\frac{2}{q - 1}} + \|r_0\|^2 \bigr)^{\frac{q}{2}}
            - \frac{q \langle r_0, r - r_0 \rangle }{\bigl( \varepsilon^{\frac{2}{q - 1}} + \|r_0\|^2 \bigr)^{1 - \frac{q}{2}}} \biggr|
            \leq A_4 \|r - r_0\|^2
        \end{displaymath}
        for all $r$, $r_0 \in B_R$.
        Therefore, taking into account that
        \begin{displaymath}
            \|a(\cdot \mid \tau, y(\cdot)) - a_\ast(\cdot)\|_\infty
            \leq \max_{\xi \in [0, \tau]} \|y(\xi)\| + \|a_\ast(\cdot)\|_\infty
            \leq R
            \quad \forall \tau \in [t, T]
        \end{displaymath}
        by~\cref{a_Lipschitz} and that $a(\cdot \mid t, x(\cdot)) = a(\cdot \mid t, y(\cdot))$ due to~\cref{a_non-anticipative}, we get, for all $\tau$, $\xi \in [t, T]$,
        \begin{align}
            & \bigg| \bigl( \varepsilon^{\frac{2}{q - 1}} + \|a(\xi \mid \tau, y(\cdot)) - a_\ast(\xi)\|^2 \bigr)^{\frac{q}{2}}
            - \bigl( \varepsilon^{\frac{2}{q - 1}} + \|a(\xi \mid t, x(\cdot)) - a_\ast(\xi)\|^2 \bigr)^{\frac{q}{2}}
            \label{lemma_proof_ci_differentiability_basic_xi} \\[0.2em]
            & \quad - \frac{q \langle a(\xi \mid t, x(\cdot)) - a_\ast(\xi), a(\xi \mid \tau, y(\cdot)) - a(\xi \mid t, x(\cdot)) \rangle }{\bigl( \varepsilon^{\frac{2}{q - 1}} + \|a(\xi \mid t, x(\cdot)) - a_\ast(\xi)\|^2 \bigr)^{1 - \frac{q}{2}}} \bigg|
            \nonumber \\[0.4em]
            & \qquad
            \leq A_4 \| a(\xi \mid \tau, y(\cdot)) - a(\xi \mid t, x(\cdot)) \|^2.
            \nonumber
        \end{align}

        In addition, consider $M > 0$ such that $\|(^C D^\alpha y)(\xi)\| \leq M$ for a.e. $\xi \in [t, T]$.
        Then, for every $\tau \in [t, T]$, in view of~\cref{I^alpha_D^alpha,a}, we derive, for all $\xi \in [t, \tau]$,
        \begin{align}
            & \|a(\xi \mid \tau, y(\cdot)) - a(\xi \mid t, x(\cdot))\|
            = \|y(\xi) - a(\xi \mid t, x(\cdot))\| \label{Delta_a_Holder_leq} \\[0.2em]
            & \hphantom{\|a(\xi \mid \tau, y(\cdot)) - a(\xi \mid t, x(\cdot))\|}
            = \frac{1}{\Gamma(\alpha)} \biggl\| \int_{t}^{\xi} \frac{(^C D^\alpha y)(\eta)}{(\xi - \eta)^{1 - \alpha}} \, \rd \eta \biggr\|
            \nonumber \\[0.2em]
            & \hphantom{\|a(\xi \mid \tau, y(\cdot)) - a(\xi \mid t, x(\cdot))\|}
            \leq \frac{M (\xi - t)^\alpha}{\Gamma(\alpha + 1)}
            \leq \frac{M (\tau - t)^\alpha}{\Gamma(\alpha + 1)}
            \nonumber
        \end{align}
        and, for all $\xi \in (\tau, T]$,
        \begin{align}
            & \|a(\xi \mid \tau, y(\cdot)) - a(\xi \mid t, x(\cdot))\|
            = \frac{1}{\Gamma(\alpha)} \biggl\| \int_{t}^{\tau} \frac{(^C D^\alpha y)(\eta)}{(\xi - \eta)^{1 - \alpha}} \, \rd \eta \biggr\|
            \label{Delta_a_Holder_geq} \\[0.2em]
            & \hphantom{\|a(\xi \mid \tau, y(\cdot)) - a(\xi \mid t, x(\cdot))\|}
            \leq \frac{M \bigl( (\xi - t)^\alpha - (\xi - \tau)^\alpha \bigr)}{\Gamma(\alpha + 1)}
            \leq \frac{M (\tau - t)^\alpha}{\Gamma(\alpha + 1)}
            \nonumber
        \end{align}
        and
        \begin{equation} \label{Delta_a_Lip_simple}
            \|a(\xi \mid \tau, y(\cdot)) - a(\xi \mid t, x(\cdot))\|
            \leq \frac{M (\tau - t)}{\Gamma(\alpha) (\xi - \tau)^{1 - \alpha}}.
        \end{equation}
        Moreover, based on the integration by parts formula, we obtain, for every $\xi \in (\tau, T]$,
        \begin{displaymath}
            a(\xi \mid \tau, y(\cdot)) - a(\xi \mid t, x(\cdot))
            =  \frac{z(\tau)}{\Gamma(\alpha) (\xi - \tau)^{1 - \alpha}}
            - \frac{1 - \alpha}{\Gamma(\alpha)} \int_{t}^{\tau} \frac{z(\eta)}{(\xi - \eta)^{2 - \alpha}} \, \rd \eta.
        \end{displaymath}
        Hence, owing to the estimate $\|z(\eta)\| \leq M (\eta - t)$ for all $\eta \in [t, T]$, we conclude that
        \begin{align}
            & \biggl\| a(\xi \mid \tau, y(\cdot)) - a(\xi \mid t, x(\cdot))
            - \frac{z(\tau)}{\Gamma(\alpha) (\xi - t)^{1 - \alpha}} \biggr\|
            \label{Delta_a_minus_z} \\[0.1em]
            & \quad \leq \biggl\| a(\xi \mid \tau, y(\cdot)) - a(\xi \mid t, x(\cdot))
            - \frac{z(\tau)}{\Gamma(\alpha) (\xi - \tau)^{1 - \alpha}} \biggr\|
            \nonumber \\[0.1em]
            & \qquad
            + \frac{M (\tau - t)}{\Gamma(\alpha)} \biggl( \frac{1}{(\xi - \tau)^{1 - \alpha}} - \frac{1}{(\xi - t)^{1 - \alpha}} \biggr)
            \nonumber \\[0.1em]
            & \quad \leq \frac{2 M (\tau - t)}{\Gamma(\alpha)} \biggl( \frac{1}{(\xi - \tau)^{1 - \alpha}} - \frac{1}{(\xi - t)^{1 - \alpha}} \biggr)
            \nonumber
        \end{align}
        for any $\xi \in (\tau, T]$.

        Now, let us fix $\tau \in (t, T - \theta]$.
        Denoting the left-hand side of inequality~\cref{lemma_proof_ci_differentiability_first_term} by $S_1$, due to~\cref{lemma_proof_ci_differentiability_basic_xi}, we get
        \begin{align}
            & S_1
            \leq A_4 \| a(T \mid \tau, y(\cdot)) - a(T \mid t, x(\cdot)) \|^2
            \label{lemma_proof_ci_differentiability_first_term_1} \\[0.2em]
            & \hphantom{S_1} \quad
            + q \bigl( \varepsilon^{\frac{2}{q - 1}} + R^2 \bigr)^{\frac{q - 1}{2}}
            \biggl\| a(T \mid \tau, y(\cdot)) - a(T \mid t, x(\cdot)) - \frac{z(\tau)}{\Gamma(\alpha) (T - t)^{1 - \alpha}} \biggr\|.
            \nonumber
        \end{align}
        According to~\cref{Delta_a_Lip_simple}, we have
        \begin{equation} \label{lemma_proof_ci_differentiability_first_term_2}
            \|a(T \mid \tau, y(\cdot)) - a(T \mid t, x(\cdot))\|
            \leq \frac{M (\tau - t)}{\Gamma(\alpha + 1) \theta^{1 - \alpha}}.
        \end{equation}
        Taking~\cref{Delta_a_minus_z} into account, we derive
        \begin{equation} \label{lemma_proof_ci_differentiability_first_term_3}
            \biggl\| a(T \mid \tau, y(\cdot)) - a(T \mid t, x(\cdot)) - \frac{z(\tau)}{\Gamma(\alpha) (T - t)^{1 - \alpha}} \biggr\|
            \leq \frac{2 M (1 - \alpha) (\tau - t)^2}{\Gamma(\alpha) \theta^{2 - \alpha}}.
        \end{equation}
        From~\cref{lemma_proof_ci_differentiability_first_term_1,lemma_proof_ci_differentiability_first_term_2,lemma_proof_ci_differentiability_first_term_3}, it follows that there is a number $A_2 > 0$ such that~\cref{lemma_proof_ci_differentiability_first_term} is valid.

        Let $S_2$ denote the left-hand side of inequality~\cref{lemma_proof_ci_differentiability_second_term}.
        Then, by~\cref{lemma_proof_ci_differentiability_basic_xi}, we obtain
        \begin{align}
            & S_2
            \leq A_4 \int_{t}^{T} \frac{\| a(\xi \mid \tau, y(\cdot)) - a(\xi \mid t, x(\cdot)) \|^2}{(T - \xi)^{(1 - \alpha - \beta) q}} \, \rd \xi
            \label{lemma_proof_ci_differentiability_second_term_1} \\[0.2em]
            & \hphantom{S_2 \ \ \,}
            + q \bigl( \varepsilon^{\frac{2}{q - 1}} + R^2 \bigr)^{\frac{q - 1}{2}}
            \int_{t}^{T} \frac{\bigl\| a(\xi \mid \tau, y(\cdot)) - a(\xi \mid t, x(\cdot))
            - \frac{z(\tau)}{\Gamma(\alpha) (\xi - t)^{1 - \alpha}} \bigr\|}
            {(T - \xi)^{(1 - \alpha - \beta) q}} \, \rd \xi.
            \nonumber
        \end{align}
        Owing to~\cref{Delta_a_Holder_leq}, we get
        \begin{equation} \label{lemma_proof_ci_differentiability_second_term_2}
            \int_{t}^{\tau} \frac{\| a(\xi \mid \tau, y(\cdot)) - a(\xi \mid t, x(\cdot)) \|^2}{(T - \xi)^{(1 - \alpha - \beta) q}} \, \rd \xi
            \leq \frac{M^2 (\tau - t)^{1 + 2 \alpha}}{(\Gamma (\alpha + 1))^2 \theta^{(1 - \alpha - \beta) q}}
        \end{equation}
        and
        \begin{align}
            & \int_{t}^{\tau} \frac{\bigl\| a(\xi \mid \tau, y(\cdot)) - a(\xi \mid t, x(\cdot))
            - \frac{z(\tau)}{\Gamma(\alpha)(\xi - t)^{1 - \alpha}} \bigr\|}{(T - \xi)^{(1 - \alpha - \beta) q}} \, \rd \xi
            \label{lemma_proof_ci_differentiability_second_term_3} \\[0.2em]
            & \quad \leq \int_{t}^{\tau} \frac{\| a(\xi \mid \tau, y(\cdot)) - a(\xi \mid t, x(\cdot))\|}{(T - \xi)^{(1 - \alpha - \beta) q}} \, \rd \xi
            \nonumber \\[0.2em]
            & \qquad + \frac{M (\tau - t)}{\Gamma(\alpha)} \int_{t}^{\tau} \frac{\rd \xi}{(\xi - t)^{1 - \alpha} (T - \xi)^{(1 - \alpha - \beta) q}}
            \nonumber \\[0.2em]
            & \quad \leq \frac{2 M (\tau - t)^{1 + \alpha}}{\Gamma (\alpha + 1) \theta^{(1 - \alpha - \beta) q}}.
            \nonumber
        \end{align}
        In view of~\cref{Beta,Delta_a_Holder_geq,Delta_a_Lip_simple}, we have
        \begin{align}
            & \int_{\tau}^{T} \frac{\| a(\xi \mid \tau, y(\cdot)) - a(\xi \mid t, x(\cdot)) \|^2}{(T - \xi)^{(1 - \alpha - \beta) q}} \, \rd \xi
            \label{lemma_proof_ci_differentiability_second_term_4} \\[0.2em]
            & \quad \leq \frac{M (\tau - t)^\alpha}{\Gamma(\alpha + 1)}
            \int_{\tau}^{T} \frac{\| a(\xi \mid \tau, y(\cdot)) - a(\xi \mid t, x(\cdot)) \|}{(T - \xi)^{(1 - \alpha - \beta) q}} \, \rd \xi
            \nonumber \\[0.2em]
            & \quad \leq \frac{M^2 (\tau - t)^{1 + \alpha}}{(\Gamma (\alpha + 1))^2}
            \int_{\tau}^{T} \frac{\rd \xi}{(\xi - \tau)^{1 - \alpha} (T - \xi)^{(1 - \alpha - \beta) q}}
            \nonumber \\[0.2em]
            & \quad \leq \frac{B(\alpha, 1 - (1 - \alpha - \beta) q) M^2 T^\alpha (\tau - t)^{1 + \alpha}}{(\Gamma (\alpha + 1))^2 \theta^{(1 - \alpha - \beta) q}}.
            \nonumber
        \end{align}
        Based on~\cref{tilde_C_1,Delta_a_minus_z}, we derive
        \begin{align}
            & \int_{\tau}^{T}
            \frac{\bigl\| a(\xi \mid \tau, y(\cdot)) - a(\xi \mid t, x(\cdot)) - \frac{z(\tau)}{\Gamma(\alpha) (\xi - t)^{1 - \alpha}} \bigr\|}
            {(T - \xi)^{(1 - \alpha - \beta) q}} \, \rd \xi
            \label{lemma_proof_ci_differentiability_second_term_5} \\[0.2em]
            & \hspace*{-2.7em} \leq \frac{2 M (\tau - t)}{\Gamma(\alpha)}
            \biggl( \int_{\tau}^{T} \frac{\rd \xi}{(\xi - \tau)^{1 - \alpha} (T - \xi)^{(1 - \alpha - \beta) q}}
            - \int_{\tau}^{T} \frac{\rd \xi}{(\xi - t)^{1 - \alpha} (T - \xi)^{(1 - \alpha - \beta) q}} \biggr)
            \nonumber \\[0.2em]
            & \hspace*{-2.7em} \leq \frac{2 M (\tau - t)}{\Gamma(\alpha)}
            \biggl( A_1 (\tau - t) + \int_{t}^{\tau} \frac{\rd \xi}{(\xi - t)^{1 - \alpha} (T - \xi)^{(1 - \alpha - \beta) q}} \biggr)
            \nonumber \\[0.2em]
            & \hspace*{-2.7em} \leq \frac{2 M A_1 (\tau - t)^2}{\Gamma(\alpha)}
            + \frac{2 M (\tau - t)^{1 + \alpha}}{\Gamma(\alpha + 1) \theta^{(1 - \alpha - \beta) q}}.
            \nonumber
        \end{align}
        Relations~\cref{lemma_proof_ci_differentiability_second_term_1,lemma_proof_ci_differentiability_second_term_2,lemma_proof_ci_differentiability_second_term_3,lemma_proof_ci_differentiability_second_term_4,lemma_proof_ci_differentiability_second_term_5} imply that there exists $A_3 > 0$ such that~\cref{lemma_proof_ci_differentiability_second_term} holds.

        Hence, to complete the proof, it remains to verify continuity of the mapping $\nabla^\alpha \mu_\varepsilon^{(\tau_\ast, y_\ast(\cdot))} \colon [0, T) \times AC^\alpha \to \mathbb{R}^n$.
        Let $(t_0, x_0(\cdot)) \in [0, T) \times AC^\alpha$ and let a sequence $\{(t_i, x_i(\cdot))\}_{i \in \mathbb{N}} \subset [0, T] \times AC^\alpha$ be such that $|t_i - t_0| + \|x_i(\cdot) - x_0(\cdot)\|_\infty \to 0$ as $i \to \infty$.
        Note that we can fix $\theta \in (0, T - t_0)$ and suppose that $t_i \in [0, T - \theta]$ for all $i \in \mathbb{N} \cup \{0\}$.
        For every $i \in \mathbb{N} \cup \{0\}$, consider the function
        \begin{displaymath}
            g_i(\xi)
            \triangleq \frac{q( a(\xi \mid t_i, x_i(\cdot)) - a_\ast(\xi))}
            {\Gamma(\alpha) \bigl( \varepsilon^{\frac{2}{q - 1}} + \|a(\xi \mid t_i, x_i(\cdot)) - a_\ast(\xi)\|^2 \bigr)^{1 - \frac{q}{2}}},
            \quad \xi \in [0, T].
        \end{displaymath}
        Then, according to~\cref{lemma_nu_properties_item_c_derivative_x}, we have, for all $i \in \mathbb{N} \cup \{0\}$,
        \begin{equation} \label{lemma_proof_nabla_via_g_1}
            \nabla^\alpha \mu_\varepsilon^{(\tau_\ast, y_\ast(\cdot))} (t_i, x_i(\cdot))
            = \frac{g_i(T)}{(T - t_i)^{1 - \alpha}}
            + \int_{t_i}^{T} \frac{g_i(\xi)}{(\xi - t_i)^{1 - \alpha} (T - \xi)^{(1 - \alpha - \beta) q}} \, \rd \xi.
        \end{equation}
        Observe that, continuity of mapping~\cref{a_mapping} yields
        \begin{equation} \label{lemma_proof_nabla_via_g_2}
            \|g_i(\cdot) - g_0(\cdot)\|_\infty
            \to 0
            \text{ as } i \to \infty.
        \end{equation}
        Therefore, in particular, we obtain
        \begin{equation} \label{lemma_proof_nabla_via_g_3}
            \biggl\| \frac{g_i(T)}{(T - t_i)^{1 - \alpha}} - \frac{g_0(T)}{(T - t_0)^{1 - \alpha}} \biggr\|
            \to 0
            \text{ as } i \to \infty.
        \end{equation}
        Further, for every $i \in \mathbb{N}$, taking~\cref{Beta} into account, we get
        \begin{align}
            & \biggl\| \int_{t_i}^{T} \frac{g_i(\xi)}{(\xi - t_i)^{1 - \alpha} (T - \xi)^{(1 - \alpha - \beta) q}} \, \rd \xi
            - \int_{t_0}^{T} \frac{g_0(\xi)}{(\xi - t_0)^{1 - \alpha} (T - \xi)^{(1 - \alpha - \beta) q}} \, \rd \xi \biggr\|
            \label{lemma_proof_nabla_via_g_4} \\[0.3em]
            & \hspace*{-2em} \leq B(\alpha, 1 - (1 - \alpha - \beta) q) (T - t_i)^{\alpha - (1 - \alpha - \beta) q} \|g_i(\cdot) - g_0(\cdot)\|_\infty
            \nonumber \\[0.3em]
            & \hspace*{-2em} \quad + \biggl\| \int_{t_i}^{T} \frac{g_0(\xi)}{(\xi - t_i)^{1 - \alpha} (T - \xi)^{(1 - \alpha - \beta) q}} \, \rd \xi
            - \int_{t_0}^{T} \frac{g_0(\xi)}{(\xi - t_0)^{1 - \alpha} (T - \xi)^{(1 - \alpha - \beta) q}} \, \rd \xi \biggr\|.
            \nonumber
        \end{align}
        In addition, due to~\cref{tilde_C_1}, we derive (we assume that $t_i \geq t_0$ for definiteness)
        \begin{align}
            & \biggl\| \int_{t_i}^{T} \frac{g_0(\xi)}{(\xi - t_i)^{1 - \alpha} (T - \xi)^{(1 - \alpha - \beta) q}} \, \rd \xi
            - \int_{t_0}^{T} \frac{g_0(\xi)}{(\xi - t_0)^{1 - \alpha} (T - \xi)^{(1 - \alpha - \beta) q}} \, \rd \xi \biggr\|
            \label{lemma_proof_nabla_via_g_5} \\[0.2em]
            & \hspace*{-2.3em} \leq \|g_0(\cdot)\|_\infty
            \biggl( \int_{t_i}^{T} \frac{\rd \xi}{(\xi - t_i)^{1 - \alpha} (T - \xi)^{(1 - \alpha - \beta) q}}
            - \int_{t_i}^{T} \frac{\rd \xi}{(\xi - t_0)^{1 - \alpha} (T - \xi)^{(1 - \alpha - \beta) q}} \biggr)
            \nonumber \\[0.2em]
            & \hspace*{-2.3em} \quad + \|g_0(\cdot)\|_\infty \int_{t_0}^{t_i} \frac{\rd \xi}{(\xi - t_0)^{1 - \alpha} (T - \xi)^{(1 - \alpha - \beta) q}}
            \nonumber \\[0.2em]
            & \hspace*{-2.3em} \leq \|g_0(\cdot)\|_\infty
            \biggl( A_1 (t_i - t_0) + 2 \int_{t_0}^{t_i} \frac{\rd \xi}{(\xi - t_0)^{1 - \alpha} (T - \xi)^{(1 - \alpha - \beta) q}} \biggr)
            \nonumber \\[0.2em]
            & \hspace*{-2.3em} \leq \|g_0(\cdot)\|_\infty
            \biggl( A_1 (t_i - t_0) + \frac{2 (t_i - t_0)^\alpha}{\alpha \theta^{(1 - \alpha - \beta) q}} \biggr).
            \nonumber
        \end{align}
        Thus, it follows from~\cref{lemma_proof_nabla_via_g_1,lemma_proof_nabla_via_g_2,lemma_proof_nabla_via_g_3,lemma_proof_nabla_via_g_4,lemma_proof_nabla_via_g_5} that
        \begin{displaymath}
            \|\nabla^\alpha \mu_\varepsilon^{(\tau_\ast, y_\ast(\cdot))} (t_i, x_i(\cdot))
            - \nabla^\alpha \mu_\varepsilon^{(\tau_\ast, y_\ast(\cdot))} (t_0, x_0(\cdot))\|
            \to 0
            \text{ as } i \to \infty.
        \end{displaymath}
        The lemma is proved.
    \end{proof}

    \begin{proof}[Proof of~\cref{lemma_gradient_properties}]
        Given a number $\theta \in (0, T)$, put
        \begin{equation} \label{C_3}
            C_3
            \triangleq \frac{q}{\Gamma(\alpha) \theta^{1 - \alpha}}
            \Bigl( 1 + T^{\frac{1}{q} - 1 + \alpha + \beta} \bigl( B(1 - (1 - \alpha) q, 1 - (1 - \alpha - \beta) q) \bigr)^{\frac{1}{q}} \Bigr)
        \end{equation}
        and
        \begin{equation} \label{C_4}
            C_4
            \triangleq \frac{q}{\Gamma(\alpha)} \biggl( \frac{(1 - \alpha) T^{1 - \alpha}}{\theta^{2 - \alpha}}
            + A_1 T^{1 - \alpha} + \frac{2}{\alpha \theta^{(1 - \alpha - \beta) q}} \biggr),
        \end{equation}
        where the number $A_1$ is taken from~\cref{tilde_C_1}.

        Fix $\varepsilon > 0$ and $(t, x(\cdot))$, $(\tau, y(\cdot)) \in [0, T - \theta] \times AC^\alpha$.
        According to~\cref{lemma_nu_properties_item_c_derivative_x}, we obtain
        \begin{align}
            & \| \nabla^\alpha \mu_\varepsilon^{(\tau, y(\cdot))} (t, x(\cdot)) \|
            \label{proof_d_1} \\[0.1em]
            & \quad \leq \frac{q}{\Gamma(\alpha)}
            \biggl( \frac{\bigl( \varepsilon^{\frac{2}{q - 1}} + \|a(T \mid t, x(\cdot)) - a(T \mid \tau, y(\cdot))\|^2 \bigr)^{\frac{q - 1}{2}}}
            {(T - t)^{1 - \alpha}}
            \nonumber \\[0.1em]
            & \qquad + \int_{t}^{T}
            \frac{\bigl( \varepsilon^{\frac{2}{q - 1}} + \|a(\xi \mid t, x(\cdot)) - a(\xi \mid \tau, y(\cdot))\|^2 \bigr)^{\frac{q - 1}{2}}}
            {(\xi - t)^{1 - \alpha} (T - \xi)^{(1 - \alpha - \beta) q}} \, \rd \xi \biggr).
            \nonumber
        \end{align}
        Similarly to~\cref{proof_f_1}, we have
        \begin{align}
            & \frac{\bigl( \varepsilon^{\frac{2}{q - 1}} + \|a(T \mid t, x(\cdot)) - a(T \mid \tau, y(\cdot))\|^2 \bigr)^{\frac{q - 1}{2}}}
            {(T - t)^{1 - \alpha}}
            \label{proof_d_2} \\[0.1em]
            & \quad \leq \frac{\bigl( \nu_\varepsilon(t, x(\cdot), \tau, y(\cdot)) + C_1 \varepsilon^{\frac{q}{q - 1}} \bigr)^{\frac{q - 1}{q}}}{\theta^{1 - \alpha}}.
            \nonumber
        \end{align}
        Applying H\"{o}lder's inequality, we conclude that
        \begin{align}
            & \int_{t}^{T} \frac{\bigl( \varepsilon^{\frac{2}{q - 1}} + \|a(\xi \mid t, x(\cdot)) - a(\xi \mid \tau, y(\cdot))\|^2 \bigr)^{\frac{q - 1}{2}}}
            {(\xi - t)^{1 - \alpha} (T - \xi)^{(1 - \alpha - \beta) q}} \, \rd \xi
            \label{proof_d_3} \\[0.1em]
            & \quad \leq \biggl( \int_{t}^{T} \frac{\rd \xi}{(\xi - t)^{(1 - \alpha) q} (T - \xi)^{(1 - \alpha - \beta) q}} \biggr)^{\frac{1}{q}}
            \nonumber \\[0.3em]
            & \qquad \times \biggl( \int_{t}^{T} \frac{ \bigl( \varepsilon^{\frac{2}{q - 1}} + \|a(\xi \mid t, x(\cdot)) - a(\xi \mid \tau, y(\cdot))\|^2 \bigr)^{\frac{q}{2}}}
            {(T - \xi)^{(1 - \alpha - \beta) q}} \, \rd \xi \biggr)^{\frac{q - 1}{q}}.
            \nonumber
        \end{align}
        By~\cref{Beta}, we get
        \begin{align}
            & \int_{t}^{T} \frac{\rd \xi}{(\xi - t)^{(1 - \alpha) q} (T - \xi)^{(1 - \alpha - \beta) q}}
            \label{proof_d_4} \\[0.3em]
            & \quad \leq \frac{T^{1 - (1 - \alpha - \beta) q} B(1 - (1 - \alpha) q, 1 - (1 - \alpha - \beta) q)}{\theta^{(1 - \alpha) q}}.
            \nonumber
        \end{align}
        Moreover, in view of~\cref{nu}, it holds that
        \begin{align}
            & \int_{t}^{T} \frac{ \bigl( \varepsilon^{\frac{2}{q - 1}} + \|a(\xi \mid t, x(\cdot)) - a(\xi \mid \tau, y(\cdot))\|^2 \bigr)^{\frac{q}{2}}}
            {(T - \xi)^{(1 - \alpha - \beta) q}} \, \rd \xi
            \label{proof_d_5} \\[0.3em]
            & \quad \leq \nu_\varepsilon(t, x(\cdot), \tau, y(\cdot)) + C_1 \varepsilon^{\frac{q}{q - 1}}.
            \nonumber
        \end{align}
        Putting together~\cref{proof_d_1,proof_d_2,proof_d_3,proof_d_4,proof_d_5}, we find that~\cref{gradient_bound} is fulfilled with $C_3$ from~\cref{C_3}.

        Further, based on~\cref{lemma_nu_properties_item_c_derivative_x}, similarly to~\cref{lemma_proof_nabla_via_g_5}, we derive (we assume that $\tau \geq t$ for definiteness)
        \begin{align*}
            & \| \nabla^\alpha \mu_\varepsilon^{(\tau, y(\cdot))} (t, x(\cdot))
            + \nabla^\alpha \mu_\varepsilon^{(t, x(\cdot))} (\tau, y(\cdot)) \| \\[0.3em]
            & \quad \leq \frac{q}{\Gamma(\alpha)}
            \max_{\xi \in [0, T]}
            \bigl( \varepsilon^{\frac{2}{q - 1}} + \|a(\xi \mid t, x(\cdot)) - a(\xi \mid \tau, y(\cdot))\|^2 \bigr)^{\frac{q - 1}{2}} \\[0.2em]
            & \qquad \times \biggl( \frac{1}{(T - \tau)^{1 - \alpha}} - \frac{1}{(T - t)^{1 - \alpha}}
            + A_1 (\tau - t) + \frac{2 (\tau - t)^\alpha}{\alpha \theta^{(1 - \alpha - \beta) q}} \biggr),
        \end{align*}
        wherefrom, owing to the estimate
        \begin{displaymath}
            \frac{1}{(T - \tau)^{1 - \alpha}} - \frac{1}{(T - t)^{1 - \alpha}}
            \leq \frac{(1 - \alpha) (\tau - t)}{\theta^{2 - \alpha}},
        \end{displaymath}
        we obtain that~\cref{difference_of_gradients} is valid with $C_4$ given by~\cref{C_4}.
        The lemma is proved.
    \end{proof}

\end{document}